\documentclass[11pt]{amsart}
\usepackage{amsmath,amssymb,amsthm,url}
\usepackage[latin1]{inputenc}
\usepackage{version,tabularx,multicol}
\usepackage{graphicx,float}

\newcommand{\reff}[1]{(\ref{#1})}

\theoremstyle{plain}
\newtheorem{theo}{Theorem}[section]

\newtheorem{cor}[theo]{Corollary}
\newtheorem{prop}[theo]{Proposition}

\newtheorem{lem}[theo]{Lemma}

\theoremstyle{remark}

\newcommand{\cp}{{\mathcal P}}

\newcommand{\cy}{{\mathcal Y}}

\newcommand{\E}{{\mathbb E}}

\newcommand{\N}{{\mathbb N}}
\renewcommand{\P}{{\mathbb P}}

\newcommand{\llambda}{{\Lambda}}

\newcommand{\ind}{{\bf 1}}

\newcommand{\val}[1]{\mathop{\left| #1 \right|}\nolimits}
\newcommand{\inv}[1]{\mathop{\frac{1}{ #1}}\nolimits}
\newcommand{\expp}[1]{\mathop {\mathrm{e}^{ #1}}}

\newcommand{\ixen}{X^{(n)}}
\newcommand{\igreken}{Y^{(n)}}

\begin{document}

\title{Asymptotic results on the length of  coalescent trees}

\date{\today}

 \author{Jean-Fran\c cois Delmas}
 \address{CERMICS, \'Ecole des Ponts, ParisTech, 6-8
av. Blaise Pascal, 
  Champs-sur-Marne, 77455 Marne La Vall\'ee, France.}
 \email{delmas@cermics.enpc.fr}

\author{Jean-St\'ephane Dhersin}
\address{UFR de Math\'ematiques et d'Informatique, Universit\'e Ren\'e Descartes,  45 rue des Saints P\`eres, 75270 Paris Cedex 06, France}
\email{dhersin@math-info.univ-paris5.fr}
\author{Arno Siri-Jegousse}
\address{UFR de Math\'ematiques et d'Informatique, Universit\'e Ren\'e Descartes,  45 rue des Saints P\`eres, 75270 Paris Cedex 06, France}
\email{Arno.Jegousse@math-info.univ-paris5.fr}

\begin{abstract}
  We  give  the asymptotic  distribution  of  the  length of    partial
  coalescent trees for Beta and related coalescents. This allows us to
  give the asymptotic distribution  of the number of (neutral)
  mutations in the partial tree. This is a first step to study the
  asymptotic distribution of a natural  estimator of  DNA mutation rate for
  species with large families. 

\end{abstract}



\maketitle

\section{Introduction}
\subsection{Motivations}

The Kingman coalescent, see \cite{k:c,k:p}, allows to describe the genealogy
of $n$ individuals in a Wright-Fisher  model, when the size of the whole
population is very large and time is well rescaled.  In what follows, we
consider  only neutral mutations  and the  infinite allele  model, where
each   mutation   gives  a   new   allele.    The  Watterson   estimator
\cite{w:nssgmwr},  based on  the  number of  different alleles  observed
among $n$ individuals alive today, $K^{(n)} $, allows to estimate the rate of
mutation  for  the DNA,  $\theta$.   This  estimator  is consistent  and
converges at rate $1/\sqrt{\log(n)}$.

Other  models of  population where  one individual  can produce  a large
number of children  give rise to more general  coalescent processes than
the  Kingman coalescent,  where multiple  collisions appear,  see Sagitov
\cite{s:gcamal}  and Schweinsberg \cite{s:cposgwp}  (such models  may be
relevant  for  oysters and  some  fish  species
\cite{bbb:mdviipopo,ew:cpwtdoonaiihs}).   In Birkner  and  al. 
\cite{bbcemsw} and in Schweinsberg \cite{s:cposgwp} a natural family of
one parameter  coalescent processes arise  to describe the  genealogy of
such populations: the Beta  coalescent with parameter $\alpha\in (1,2)$.
Results  from Beresticky  and  al.  \cite{bbs:stpbc}  give a  consistent
estimator,  based  on the  observed  number,  $K^{(n)}  $, of  different
alleles for the rate $\theta$ of mutation of DNA.  This paper is a first
step to study the convergence  rate of this estimator or equivalently to
the study the  asymptotic distribution of $K^{(n)} $.   Results are also
known for the asymptotic distribution of $K^{(n)} $ for other coalescent
processes,   see    Drmota   and   al.     \cite{dimr:rrc}   and   M\"ohle
\cite{m:nssplfs}.

For the Beta  coalescent, the asymptotic distribution of  $K^{(n)} $ depends on
$\theta$  but also  on the  parameter $\alpha$.   In particular,  if the
mutation rate  of the DNA is  known, the asymptotic  distribution of $K^{(n)} $
allows to deduce  an estimation and a confidence  interval for $\alpha$,
which in a sense characterize the  size of a typical family according to
\cite{s:cposgwp}.

\subsection{The coalescent tree and mutation rate}
We consider  at time $t=0$  a number, $n\geq  1$ of individuals,  and we
look backward  in time. Let  $\cp_n$ be the  set of partitions  of $\{1,
\ldots, n\}$.  For $t\geq 0$, let $\Pi^{(n)}_t$ be an element of $\cp_n$
such  that  each  block  of  $\Pi^{(n)}_t$  corresponds  to  the  initial
individuals which have a common ancestor at time $-t$.  We assume that if
we   consider  $b$  blocks,   $k$  of   them  merge   into  1   at  rate
$\lambda_{b,k}$,  independent of  the current  number of  blocks.  Using
this property and the compatibility relation implied when one consider a
larger  number of  initial  individuals, Pitman  \cite{p:cmc}, see  also
Sagitov  \cite{s:gcamal}  for a  more  biological  approach, showed  the
transition rates are given by
\[
\lambda_{b,k}=\int_{(0,1)} x^{k-2} (1-x)^{b-k} \Lambda(dx),\quad 2\leq
k\leq b, 
\]
for some  finite measure $\Lambda$  on $[0,1]$, and that  $\Pi^{(n)}$ is
the restriction of  the so-called coalescent process defined  on the set
of partitions of $\N^*$.  The  Kingman coalescent correspond to the case
where  $\Lambda$  is  the  Dirac   mass  at  $0$,  see  \cite{k:c}.   In
particular, in the Kingman coalescent,  only two blocks merge at a time.
The  Bolthausen-Sznitman \cite{bs:rpcacm}  coalescent correspond  to the
case  where   $\Lambda$  is  the   Lebesgue  measure  on   $[0,1]$.  The
Beta-coalescent  introduced in  Birkner  and al.  \cite{bbcemsw} and  in
Schweinsberg   \cite{s:cposgwp},   see   also   Bertoin  and   Le   Gall
\cite{blg:sfcp3} and Beresticky and al.  \cite{bbs:bcsrt} , corresponds
to $\Lambda(dx) =C_0 x^{\alpha-1}(1-x)^{1-\alpha} \ind_{(0,1)}(x) \; dx$
for some constant $C_0>0$.

Notice $\Pi^{(n)}=(\Pi^{(n)}_t,t\geq 0)$ is a Markov process starting at
the  trivial partition  of $\{1,  \ldots,  n\}$ into  $n$ singletons.   We
denote by $R_t^{(n)}$ the number of blocks of $\Pi^{(n)}_t$, that is the
number of  common ancestors alive at  time $-t$.  In  particular we have
$R^{(n)}_0=n$.  We shall  omit the  superscript $(n)$  when there  is no
confusion.  The process $R=(R_t,t\geq  0)$ is  a continuous  time Markov
process  taking values  in $\N^*$.   The number  of possible  choices of
$\ell+1$ blocks  among $k$ is $\binom{k }{\ell+1}$  (for $1\leq \ell\leq
k-1$)   and   each   group    of   $\ell+1$   blocks   merge   at   rate
$\lambda_{k,\ell+1}$.  So the waiting  time of  $R$ in  state $k$  is an
exponential random variable with parameter
\begin{equation}
   \label{eq:gk}
g_k=\sum_{\ell=1}^{k -1} \binom{k }{\ell+1}
\lambda_{k ,\ell+1}=\int_{(0,1)} \Big(1- 
(1-x)^k  -k  x(1-x)^{k -1} \Big) \frac{\Lambda(dx)}{x^2}
\end{equation}
and  is distributed as $E/g_k$, where $E$ is an exponential random
variable with mean $1$. 

The apparition time of the
most recent common ancestor (MRCA) is $T_n=\inf\{t>0; R_t=1\}$.

Let $Y=(Y_k, k\geq 1)$ be the different states of the process $R$. It is
defined  by  $Y_0=R_0$  and  for  $k\geq 1$,  $Y_k=R_{S_k}$,  where  the
sequence  of jumping  time $(S_k,  k\geq 0)$  is defined  inductively by
$S_0=0$ and  for $k\geq  1$, $S_k=\inf\{t>S_{k-1}; R_t\neq  R_{S_{k-1}}\}$. We
use the convention that $\inf \emptyset=+\infty $ and $Y_k=1$ for $k\geq
\tau_n$,  where $\tau_n=\inf\{k; R_{S_k}=1\}$  is the  number of
jumps of the process $R$ until it reach the absorbing state $1$. The
number $\tau_n$ is the number of coalescences.

We shall write $Y^{(n)}$ instead of $Y$ when it will be convenient to
stress that $Y$ starts at time 0 at point $n$. Notice $Y$ is an
$\N^*$-valued  discrete
time Markov
chain, with probability transition
\begin{equation}
   \label{eq:P}
P(k,k-\ell)=\frac{\binom{k}{\ell+1}
\lambda_{k,\ell+1}}{g_k}. 
\end{equation}

The
sum of the lengths of all branches in the coalescent tree until the MRCA
is distributed as  
\[
L^{(n)}=\sum_{k=0}^{\tau_n-1} \frac{Y^{(n)}_{k} }{g_{Y^{(n)}_{k} }}
  E_k,
\]
where $(E_k, k\geq 0)$ are independent exponential random variables with
expectation $1$.

In the infinite allele model, one assume that (neutral) mutations appear
in the genealogy at random with rate $\theta$.  In particular by looking
at the  number $K^{(n)} $ of  different alleles among $n$  individuals, one
get  the number of  mutations which  occured in  the genealogy  of those
individuals  after  the most  recent  common  ancestor.  In  particular,
conditionally  on the  length  of the  coalescent  tree $L^{(n)}$,  the
number $K^{(n)} $ of  mutations is  distributed according to a  Poisson r.v.
with  parameter   $\theta  L^{(n)}$.   Therefore,  we   have  that
$\displaystyle   \frac{K^{(n)}    -\theta  L^{(n)}}{\sqrt{\theta   L^{(n)}}}$
converges in distribution  to a standard Gaussian  r.v.  (with mean 0
and variance 1).  If the asymptotic distribution of $L^{(n)}$ is known,
one can deduce the asymptotic distribution of $K^{(n)} $.

\subsection{Known results}

\subsubsection{Kingman   coalescence}.   For   Kingman   coalescence,  a
coalescence corresponds to the apparition  of a common ancestor of only
two  individuals.  In  particular,   we  have  for  $0\leq  k\leq  n-1$,
$Y_k^{(n)}=n-k$.  Thus we  get $\tau_n=n-1$  as well  as $g_{Y^{(n)}_k}=
(n-k)(n-k-1)/2$.    We  also   have   $\displaystyle  \frac{L^{(n)}}{2}=
\sum_{k=0}^{n-2}  \inv{n-k-1} E_k= \sum_{k=1}^{n-1}  \inv{k} E_{n-k-1}$.
The  r.v.   $L^{(n)}/2$  is  distributed  as  the   sum  of  independent
exponential r.v.  with parameter $1$ to $n-1$, that is as the maximum on
$n-1$  independent   exponential  r.v.    with  mean  $1$,   see  Feller
\cite{f:ipta}   section   I.6.    An   easy   computation   gives   that
$L^{(n)}/(2\log(n))$   converges  in   probability  to   $1$   and  that
$\displaystyle \frac{L^{(n)}}{2}  - \log(n) $  converges in distribution
to  the  Gumbel distribution  (with  density  $\displaystyle \expp{-x  -
  \exp{-x}}$) when $n$ goes to infinity.  It is then easy to deduce that
$\displaystyle     \frac{K^{(n)}    -\theta    \E[L^{(n)}]}{\sqrt{\theta
    \E[L^{(n)}]}}$ converges in distribution to the standard Gaussian
distribution.  This provides 
the  weak convergence  and  the asymptotic  normality  of the  Watterson
\cite{w:nssgmwr}  estimator  of  $\theta$: $\displaystyle  \frac{K^{(n)}
}{\E[L^{(n)}]}= \frac{K^{(n)} }{\sum_{k=1}^{n-1} \inv{k}}$. See also the
appendix in \cite{dimr:rrc}.

\subsubsection{Bolthausen-Sznitman coalescence}

In Drmota and al. \cite{dimr:rrc}, the authors consider the Bolthausen-Sznitman
coalescence: $\Lambda$ is the Lebesgue measure on $[0,1]$. In this case
they prove that $\displaystyle \inv{n} \log(n) L^{(n)}$ converges in
probability to $1$ and that  $\displaystyle  \frac{L^{(n)} - a_n}{b_n}$
converges in distribution to a stable r.v. $Z$ with Laplace transform 
$\displaystyle  \E[\expp{-\lambda Z}]=\expp{\lambda \log(\lambda)} $ for
$\lambda>0$, where 
\[
a_n=\frac{n}{\log(n)} + \frac{n \log(\log(n)) }{\log(n)^2}
\quad\text{and}\quad b_n=\frac{n}{\log(n)^2}.
\]
It is then easy to deduce that $\displaystyle \frac{ K^{(n)} - \theta
a_n}{\theta b_n}$ converges to $Z$. 

\subsubsection{The case $\int_{(0,1]} x^{-1}
\Lambda(dx) <\infty $}

In M\"ohle \cite{m:nssplfs}, the author investigates the case where $x^{-1}
\Lambda(dx)$ is a finite measure and consider directly the asymptotic
distribution of $K^{(n)}$. In particular he gets that
$K^{(n)}/n\theta$ converges in distribution to a non-negative r.v. $Z$
uniquely determined by its moments: for $k\geq 1$, 
\[
\E[Z^k]=\frac{k!}{\prod_{i=1}^k \Phi(i)},\quad\text{with} \quad
\Phi(i)=\int_{[0,1]} (1- (1-x)^i) x^{-2} \Lambda(dx). 
\]
There is an equation in law for $Z$ when $\Lambda$ is a simple
measure, that is when $\int_{(0,1]} x^{-2} \Lambda(dx)<\infty $. 

\subsubsection{Beta coalescent}
The Beta-coalescent correspond to the case where $\Lambda$ is the
Beta$(2-\alpha,\alpha)$ distribution, with $\alpha\in (1,2)$:
$\displaystyle
\Lambda(dx)=\inv{\Gamma(2-\alpha) \Gamma(\alpha)}
  x^{1-\alpha}(1-x)^{\alpha-1} dx$. The Kingman coalescent can be viewed
  as the asymptotic case $\alpha=2$ and the Bolthausen-Sznitman
  coalescence as the asymptotic case $\alpha=1$. 

The first order asymptotic behavior of $L^{(n)}$ is given in
\cite{bbs:stpbc}, theorem 1.9:  $n^{\alpha-2} L^{(n)}$ converges in
probability to $\displaystyle \frac{\Gamma(\alpha)
  \alpha(\alpha-1)}{2-\alpha}$. We shall now investigate the asymptotic
distribution of $L^{(n)}$. 

\subsection{Main result}
In this paper we shall state a partial result concerning the asymptotic
distribution of $L^{(n)}$. 
We shall only give the asymptotic distribution of the total length of the
coalescent tree up to the $\lfloor nt \rfloor$-th coalescence:
\begin{equation}
   \label{eq:Lnt}
L^{(n)}_{t}=\sum_{k=0}^{\lfloor nt \rfloor \wedge (\tau_n-1)}
\frac{Y^{(n)}_{k} }{g_{Y^{(n)}_{k} }} 
  E_k,
\end{equation}
where $\lfloor x \rfloor$ is the largest integer smaller or equal to $x$
for $x\geq 0$. 

We  say
$g=O(f)$, where  $f$ is  a non-negative function  and $g$ a  real valued
function  defined on a  set $E$  (mainly here  $E=[0,1]$ or  $E=\N^*$ or
$E=\N^*\times [0,1]$), if there exists a finite constant $C>0$ such that
$|g(x)|\leq Cf(x)$ for all $x\in E$.

Let   $\nu(dx)=x^{-2}\Lambda(dx)$  and   $\rho(t)=\nu((t,1])$.   
We  assume  that   $\rho(t)=C_0t^{-\alpha}  +O(t^{-\alpha+\zeta})$  for  some
$\alpha\in (1,2)$, $C_0>0$  and  $\zeta>1-1/\alpha$.  This includes the
Beta$(2-\alpha,\alpha)$ distribution for $\Lambda$. We  have, see Lemma
\ref{lem:dlg},  that 
\[
g_n=C_0 \Gamma(2-\alpha) n^\alpha +  O(n^{\alpha -\min(\zeta,1)}).
\]

Let $\gamma=\alpha-1$. Let $V=(V_t,t\geq  0)$ be a $\alpha$-stable L\'evy
process  with no  positive jumps  (see chap.   VII in  \cite{b:pl}) with
Laplace   exponent  $\psi(u)=u^\alpha/\gamma$:   for   all  $u\geq   0$,
$\E[\expp{-uV_t}]=\expp{t u^\alpha/\gamma}$.

We  first  give in
Proposition \ref{prop:tau_n}   the
asymptotic for the number of  coalescences, $\tau_n$:
\[
n^{-\inv{\alpha}}\left(n-\frac{\tau_n
  }{\gamma}\right)
\; \xrightarrow[n\rightarrow \infty ]{\text{(d)}} \; 
V_\gamma.
\]
See  also Gnedin  and  Yakubovich \cite{gy:ncl}  and  Iksanov and  M\"ohle
\cite{im:rrratdmc} for  different proofs of this  results under slightly
different or stronger hypothesis.  Then  we give the asymptotics of $\hat
L^{(n)}_t $  defined as $  C_0\Gamma(2-\alpha) L^{(n)}_t$ but  for the
exponential r.v.  $E_k$  which are replaced by their mean  that is 1 and
for   $g_{Y^{(n)}_{k}   }$  which   is   replaced   by  its   equivalent
$C_0\Gamma(2-\alpha) \left(Y^{(n)}_{k}\right)^{2-\alpha}$:
\begin{equation}
   \label{eq:defLhat}
\hat L^{(n)}_{t}= \sum_{k=0}^{\lfloor nt \rfloor \wedge (\tau_n-1)}
\left(Y^{(n)}_{k}\right)^{1-\alpha}.
\end{equation}
For $t\in [0,\gamma]$, we set 
\[
v(t) =\int_0^t
    \left(1-\frac{r}{\gamma}\right)^{-\gamma} dr.
\] 
Theorem  \ref{theo:main} gives that the following
convergence in distribution holds for all $t\in (0,\gamma)$
\begin{equation}
   \label{eq:CVLNDV-intro}
n^{-1+\alpha -1/\alpha} (  \hat L^{(n)}_t-n^{2-\alpha}
v(t))  \; \xrightarrow[n\rightarrow \infty ]{\text{(d)}}
\; (\alpha-1) \int_0^t dr\; (1-\frac{r}{\gamma})^{-\alpha} V_r.
\end{equation}
Then we deduce our main result, Theorem \ref{th:ConvLn}.
Let $\displaystyle \alpha\in (1, \frac{1+\sqrt{5}}{2})$.    
Then for all $t\in (0,\gamma)$, we have   the following convergence in
 distribution 
\begin{equation}
   \label{eq:CVlNDV_intro}
n^{-1+\alpha -1/\alpha} \left(  L^{(n)}_t-n^{2-\alpha}
\frac{v(t)}{C_0 \Gamma(2-\alpha)} \right)  \; \xrightarrow[n\rightarrow \infty
]{\text{(d)}} 
\;  \frac{\alpha-1 }{C_0 \Gamma(2-\alpha)} \int_0^t dr\;
(1-\frac{r}{\gamma})^{-\alpha} V_r.  
\end{equation}
We also  have that $n^{\alpha-2} L^{(n)}_t$ converges  in probability to
$\displaystyle \frac{v(t)}{C_0 \Gamma(2-\alpha)}$ for $\alpha\in (1,2)$.
For $t=\gamma$, intuitively we  have $L^{(n)}_\gamma$ close to $L^{(n)}$
as  $\tau_n$ is  close to  $n/\gamma$. In  particular, one  expects that
$n^{\alpha-2}  L^{(n)}$  converges   in  probability  to  $\displaystyle
\frac{v(\gamma)}{C_0   \Gamma(2-\alpha)}$.   For   the  Beta-coalescent,
$\displaystyle     \Lambda(dx)=\inv{\Gamma(2-\alpha)     \Gamma(\alpha)}
x^{1-\alpha}(1-x)^{\alpha-1} dx$, we have $C_0=1/\alpha \Gamma(2-\alpha)
\Gamma(\alpha)$ and  indeed, theorem 1.9 in  \cite{bbs:stpbc} gives that
$n^{\alpha-2}  L^{(n)}$  converges   in  probability  to  $\displaystyle
\frac{\Gamma(\alpha)  \alpha(\alpha-1)}{2-\alpha}=  \frac{v(\gamma)}{C_0
  \Gamma(2-\alpha)}$.  Notice theorem  1.9 in \cite{bbs:stpbc} is stated
for more general coalescents than the Beta-coalescent.

In Corollary \ref{cor:main}, we  give the asymptotic distribution of the
number  $K^{(n)}_t$  of mutations  on  the  coalescent  tree up  to  the
$\lfloor  nt   \rfloor$-th  coalescent  for  $\alpha   \in  (1,2)$.   In
particular,   for  $\displaystyle  \alpha>   \frac{1+\sqrt{5}}{2}$,  the
approximations of the exponential r.v.  by their mean are more important
than the fluctuations of $\hat L^{(n)}$, and the asymptotic distribution
is gaussian.

\subsection{Organization of  the paper}  

In Section \ref{sec:jump}  we give
estimates (distribution, Laplace transform) for the number of collisions
in the first coalescence in a population of $n$ individuals.
We prove the asymptotic distribution of the number of collisions,
$\tau_n$, in Section \ref{sec:taun}, as well as an invariance principle
for the coalescent process $Y^{(n)}$, see Corollary \ref{cor:cvVn}. 
In Section \ref{sec:prelim}, we give error bounds on the approximation
of $L^{(n)}_t$ by $\hat L^{(n)}_t/ C_0 \Gamma(2-\alpha)$. 
Section \ref{sec:hatLn} is devoted to the asymptotic distribution of $\hat
L^{(n)}_t$. 
Eventually,  our main result, Theorem \ref{th:ConvLn}, on the asymptotic
distribution of $L^{(n)}_t$,  and Corollary \ref{cor:main}, on the
asymptotic distribution of the number of mutations $K^{(n)}_t$,  
 and their  proofs are
given in Section \ref{sec:main}.

In what  follows, $c$ is a  non important constant which  value may vary
from line to line. 

\section{Law of the first jump}\label{sec:jump}

Let  $Y$ be  a  discrete time  Markov  chain on  $\N^*$ with  transition
kernel $P$   given   by  \reff{eq:P}   and  started  at   $Y_0=n$.   Let
$X^{(n)}_k=Y_{k-1} -Y_k$  for $k\geq 1$.  We give some estimates  on the
moment of $X^{(n)}_1$ and its Laplace transform.

For $n\geq 1$, $x\in (0,1)$, let $B_{n,x}$ be a binomial r.v. with
parameter $(n,x)$. Recall that for $1\leq k\leq n$, we have
\begin{equation}
   \label{eq:pBnx}
\P(B_{n,x}\geq  k)=\frac{n!}{(k-1)!(n-k)!} \int_0^x t^{k-1} (1-t)^{n-k} \; dt.
\end{equation}
Recall  that $\nu(dx)=x^{-2}  \Lambda(dx)$ and
$\rho(t)=\nu((t,1])$. Use the first equality in \reff{eq:gk} and
\reff{eq:pBnx} to get 
\begin{align}
g_n
\nonumber
&=\int_0^1 \sum_{k=2}^n \binom{n}{k} x^k (1-x)^{n-k} \nu(dx)\\
\nonumber
&=\int_0^1 \P(B_{n,x} \geq 2) \nu(dx)\\
   \label{eq:gn}
&=n(n-1) \int_0^1 (1-t)^{n-2} t
\rho(t) \; dt. 
\end{align}

Notice also that $\displaystyle \P(X^{(n)}_1=k)=P(n,n-k)=\inv{g_n} \int_0^1
\P(B_{n,x}=k+1) \nu(dx)$ and thus
\begin{equation}
   \label{eq:bfxn}
\P(X^{(n)}_1\geq k)=\frac{\int_0^1 \P(B_{n,x} \geq k+1)
  \nu(dx)}{g_n}=\frac{(n-2)!}{k!(n-k-1)!} \frac{\int_0^1
(1-t)^{n-k-1} t^k \rho(t) \; dt} {\int_0^1
(1-t)^{n-2} t \rho(t) \; dt}  .
\end{equation}

Let $\alpha\in (1,2)$ and $\gamma=\alpha-1$. 

We  say
$g=o(f)$, where  $f$ is  a non-negative function  and $g$ a  real valued
function  defined on  $(0,1]$, if for any $\varepsilon>0$, there exists
$x_0>0$ s.t. 
$|g(x)|\leq \varepsilon f(x)$ for all $x\in (0,x_0]$.

\begin{lem}\label{lem:cvloiXn}
   Assume  that $\rho(t)=C_0 t^{-\alpha}+o(t^{-\alpha})$. Then
   $(X_1^{(n)}, n\geq 2)$ converges in distribution  to the r.v. $X$
  such that  for all $k\geq 1$, 
\[
\P(X\geq k)=\inv{ \Gamma(2-\alpha) }
\frac{\Gamma(k+1-\alpha) }{k!}.
\]
We have $\E[X]=1/\gamma$, $\E[X^2]=+\infty $  and its Laplace transform
$\phi$ is given by:  for $u\geq 0$, 
\[
\phi(u)=\E[\expp{-uX}]=1+\frac{\expp{u} -1}{\alpha-1}
\left[(1-\expp{-u})^{\alpha-1} -1\right]. 
\]
\end{lem}

We shall use repeatedly the identity of the  beta distribution: for
$a>0$ and $b>0$, we have 
\begin{equation}
   \label{eq:gamma-rel}
\int_0^1 t^{a-1}(1-t)^{b-1}
dt=\frac{\Gamma(a)\Gamma(b)}{\Gamma(a+b)}.
\end{equation}

\begin{proof}
The condition    $\rho(t)=C_0 t^{-\alpha}+o(t^{-\alpha})$ implies that
for fixed $k\geq 1$, as $n$ goes to infinity, we have 
\[
   \int_0^1
(1-t)^{n-k-1} t^k \rho(t) \; dt 
=  \frac{\Gamma(k+1-\alpha)\Gamma(n-k)}{\Gamma(n+1-\alpha)} \left(C_0+
o(1) \right).
\]
Therefore, we get that 
\begin{align*}
\lim_{n\rightarrow\infty }\P(X_1^{(n)} \geq k)
&=\lim_{n\rightarrow\infty } \frac{(n-2)!}{k!(n-k-1)!} \frac{\int_0^1
(1-t)^{n-k-1} t^k \rho(t) \; dt} {\int_0^1
(1-t)^{n-2} t \rho(t) \; dt} \\
&=   \lim_{n\rightarrow\infty } \frac{(n-2)!}{k!(n-k-1)!}
\frac{\Gamma(k+1-\alpha)\Gamma(n-k)}{\Gamma(n+1-\alpha)} 
\frac{\Gamma(n+1-\alpha)}{\Gamma(2-\alpha)\Gamma(n-1)}\\
&=\inv{ \Gamma(2-\alpha) }
\frac{\Gamma(k+1-\alpha) }{k!}.
\end{align*}
This ends the first part of the Lemma. Notice that 
\[
\P(X\geq k)=\inv {\Gamma(\alpha)\Gamma(2-\alpha)} \int_0^1 t^{k-\alpha}
(1-t)^{\alpha-1} dt
\]
and as $\P(X=k)
=\P(X\geq k) - \P(X\geq k+1)$, we get 
\begin{equation}
   \label{eq:PX=k}
\P(X=k)
=\inv {\Gamma(\alpha)\Gamma(2-\alpha)} \int_0^1 t^{k-\alpha}
(1-t)^{\alpha} dt
=\frac{\alpha}{ \Gamma(2-\alpha) }
\frac{\Gamma(k+1-\alpha) }{(k+1)!}. 
\end{equation}

We have 
\begin{align*}
   \E[X]
=\sum_{k\geq 1} \P(X\geq k)
&=\inv {\Gamma(\alpha)\Gamma(2-\alpha)}
\int_0^1  \sum_{k\geq 1}t^{k-\alpha}
(1-t)^{\alpha-1} dt\\
&=\inv {\Gamma(\alpha)\Gamma(2-\alpha)}
\int_0^1  t^{1-\alpha}
(1-t)^{\alpha-2} dt\\
&=\inv
{\Gamma(\alpha)\Gamma(2-\alpha)}\frac{\Gamma(2-\alpha)\Gamma(\alpha-1)}
{\Gamma(1)}\\
&=\inv{\alpha-1}. 
\end{align*}
The asymptotic expansion 
\begin{equation}
\label{eq:expansionGamma}
\Gamma(z)
=\sqrt{2\pi}z^{z-1/2}\expp{-z}\left(1+\inv{12z}+o\left(\inv{z}\right)\right)
\end{equation}
implies  $\displaystyle \P(X=k) \sim_{+\infty } \frac{{\alpha}}{{\Gamma}(2-{\alpha})}k^{-\alpha-1}$. Therefore we have 
$\E[X^2]=+\infty$. We compute the Laplace transform of $X$. Let $u\geq
0$, we have 
\begin{align*}
\phi(u)
=\E[\expp{-u X}]
&=\frac{\alpha}{ \Gamma(2-\alpha) } \sum_{k\geq 1}  \inv{(k+1)!} \expp{-ku}
\int_0^\infty  x^{k-\alpha}\expp{-x}\; dx\\
&=\frac{\alpha\expp{u} }{ \Gamma(2-\alpha) } 
\int_0^\infty  \sum_{k\geq 2}  \inv{k!} \expp{-ku} x^{k-1
  -\alpha}\expp{-x}\; dx\\  
&=\frac{\alpha\expp{u}}{\Gamma(2-\alpha)}\int_0^\infty x^{-1-\alpha}\expp{-x}(\expp{x\expp{-u}}-x\expp{-u}-1)\; dx\\
&=1+\frac{\expp{u} -1}{\alpha-1}
\left[(1-\expp{-u})^{\alpha-1} -1\right],
\end{align*}
where we used \reff{eq:PX=k} with $\displaystyle \Gamma(k+1-\alpha)=\int_0^\infty
x^{k-\alpha} \expp{-x} dx$  for the first equality and  two integrations
by parts for the last.
\end{proof}

We give bounds  on $g_n$. 
\begin{lem}
\label{lem:dlg}
   Assume  that $\rho(t)=C_0t^{-\alpha} +O(t^{-\alpha+\zeta})$  for some
  $C_0>0$ and $\zeta>0$. Then we have,  for $n\geq 2$, 
\begin{equation}\label{eq:dlg}
 g_n = C_0 \Gamma(2-\alpha) n^\alpha +  O(n^{\alpha-\min(\zeta,1)}).
\end{equation}
\end{lem}
\begin{proof}
Notice that 
\[
g_n=n(n-1) \int_0^1 (1-t)^{n-2} t
\left(C_0 t^{-\alpha} + O(t^{-\alpha+\zeta})\right) \; dt
=C_0 n(n-1) \frac{\Gamma(2-\alpha)\Gamma(n-1)}{\Gamma(n+1-\alpha)} + 
h_n,
\]
where $\displaystyle h_n=n(n-1) \int_0^1 (1-t)^{n-2} t^{-\alpha+\zeta+1}
O(1) \; 
dt$. In particular, using \reff{eq:expansionGamma}, we have for $n\geq 2$
\[
|h_n|\leq c n(n-1) \int_0^1 (1-t)^{n-2} t^{-\alpha+\zeta+1} =c n(n-1)
\frac{\Gamma(2-\alpha+\zeta)\Gamma(n-1)}{\Gamma(n+1-\alpha+\zeta)}  \leq c
n^{\alpha-\zeta}. 
\]
Using      \reff{eq:expansionGamma}     again,      we      get     that
$\Gamma(n-1)/\Gamma(n+1-\alpha)= n^{\alpha-2}  + O(n^{\alpha-3})$.  This
implies that
\[
g_n= C_0 \Gamma(2-\alpha) n^{\alpha} + O(n^{\max(\alpha-1, \alpha-\zeta)}). 
\]
\end{proof}

We give an expansion of the first moment of $X_1^{(n)}$. 
\begin{lem}\label{lem:M1}
  Assume  that $\rho(t)=C_0t^{-\alpha}  +O(t^{-\alpha+\zeta})$  for some
  $C_0>0$ and $\zeta>0$.  Let  $\varepsilon_0>0$. We set 
\begin{equation}
   \label{eq:fn}
\varphi_n=\begin{cases}
{n^{-\zeta} } &\quad \text{if}\quad \zeta<\alpha-1,\\
 {n^{1-\alpha+\varepsilon_0}} &\quad \text{if}\quad \zeta=\alpha-1,\\
 n^{1-\alpha} &\quad \text{if}\quad \zeta>\alpha-1.
\end{cases}
\end{equation}
There exists a
  constant $C_{\ref{eq:M1}}$ s.t. for all $n\geq 2$, we have
\begin{equation}
   \label{eq:M1}
\val{\E[X^{(n)}_1] - \inv{\gamma} } \leq
C_{\ref{eq:M1}}\varphi_n.
\end{equation}
\end{lem}
\begin{proof}
   We have
\begin{align}
\E[X^{(n)}_1]
=\sum_{k\geq 1} \P(X^{(n)}_1\geq k)
&=\frac{\int_0^1 \sum_{k\geq 1}  \P(B_{n,x} \geq k+1) \nu(dx)}{g_n} \nonumber\\
&=\frac{\int_0^1
( \E[B_{n,x}] - \P(B_{n,x} \geq 1)) \nu(dx)}{g_n} 
\label{eq=E[Xn1}\\
&=\frac{\int_0^1 nx \nu(dx) - \int_0^1(1-(1-x)^n){\nu}(dx) }{g_n}\nonumber
\\
&= \frac{n\int_0^1 [1 -(1-t)^{n-1}]\rho(t) \; dt}{g_n}
\label{eq:E[Xn1]prime} \\
&=\frac{\int_0^1 (1-t)^{n-2}\left(\int_t^1 \rho(r) \; dr\right)\; dt}
{\int_0^1 (1-t)^{n-2} t \rho(t) \; dt }, \nonumber
\end{align}
using \reff{eq:bfxn} for the first equality and \reff{eq:gn} for the last.
Notice that 
\begin{align*}
    \int_t^1 \rho(r) \; dr
&= \inv{\gamma} t\rho(t)+ O(1)+ \int_t^1
 O(r^{-\alpha+\zeta})\; dr + 
  O(t^{-\alpha+\zeta+1})\\
&=\inv{\gamma} t\rho(t)+ O(t^{\min(-\alpha+\zeta+1,0)})+
  O(|\log(t)|)\ind_{\{\alpha-\zeta=1\}}\\
&=\inv{\gamma} t\rho(t)+
O(t^{\min(-\alpha+\zeta+1,0)})+O(t^{-\varepsilon_0})
\ind_{\{\alpha-\zeta=1\}}.   
\end{align*}
This implies that 
\[
\E[X^{(n)}_1]
=\inv{\gamma} +  \frac{n(n-1)}{g_n}\int_0^1
(1-t)^{n-2}\left(O(t^{\min(-\alpha+\zeta+1,0)})+
  O(t^{-\varepsilon_0})\ind_{\{\alpha-\zeta=1\}} \right)\; dt.
\]
Using \reff{eq:gamma-rel}, \reff{eq:expansionGamma} and Lemma
\ref{lem:dlg}, we get  
\begin{align*}
\left|{\mathbb E}[X^{(n)}_1]-\inv{\gamma}\right|
&\leq 
c \frac{n(n-1)}{g_n} \int_0^1(1-t)^{n-2}
\left(t^{\min(-\alpha+\zeta+1,0)}+
  t^{-\varepsilon_0}\ind_{\{\alpha-\zeta=1\}} \right)\;
\; dt \\
&\leq  c n^{2-\alpha}(n^{-1-\min(-\alpha+\zeta+1,0)}+
n^{-1+\varepsilon_0}\ind_{\{\alpha-\zeta=1\}} )\\
&\leq  c\varphi_n.
\end{align*}
\end{proof}

We give an upper bound for the second moment of $X_1^{(n)}$. 
\begin{lem}
  Assume  that $\rho(t)=O(t^{-\alpha})$.
  Then there exists a constant  $C_{\ref{eq:M2}}$ s.t. for all $n\geq 2$,
  we have
\begin{equation}
   \label{eq:M2}
\E\left[\left(X^{(n)}_1\right)^2\right] \leq  C_{\ref{eq:M2}} \frac{n^2}{g_n}.
\end{equation}
\end{lem}
\begin{proof}
  Using the identity $\E[Y^2]=\sum_{k\geq 1} (2k-1)\P(Y\geq k)$ for
  $\N$-valued random variables, we get 
\begin{align*}
\E\left[\left(X^{(n)}_1\right)^2\right]
&=\frac{\int_0^1
 \sum_{k\geq 1}  (2k-1) \P(B_{n,x} \geq k+1) \nu(dx)}{g_n} \\
&=\frac{\int_0^1
\left(\sum_{k\geq 1}  (2(k+1)-1) \P(B_{n,x} \geq k+1)- 2\sum_{k\geq 1}
  \P(B_{n,x} \geq k+1 )\right)  \nu(dx)}{g_n} \\
&=\frac{\int_0^1
\left(\E[B_{n,x}^2] - 2 \E[B_{n,x}] + \P(B_{n,x} \geq 1) \right)
\nu(dx)}{g_n} \\ 
&=\frac{\int_0^1
\left(\E[B_{n,x}^2] - \E[B_{n,x}] \right)
\nu(dx)}{g_n} -\E[X_1^{(n)}]\\ 
&=\frac{\int_0^1 n(n-1) x^2  \nu(dx) }{g_n} -
\E[X_1^{(n)}] \\
&= 2n(n-1) \frac{\int_0^1 t \rho(t) \; dt }{g_n}- \E[X_1^{(n)}],
\end{align*}
where  we  have  used  \reff{eq=E[Xn1}  for the  fourth  equality.   Use
$\int_0^1  t  \rho(t) \;  dt  <\infty  $  and $\E[X_1^{(n)}]\geq  0$  to
conclude.
\end{proof}

We consider $\phi_n$ the Laplace transform of $X_1^{(n)}$: for $u\geq 0$, 
$   \phi_n(u)=\E[\expp{-u X^{(n)}_1}]$.
\begin{lem}
\label{lm:phin}
   Assume  that $\rho(t)=C_0t^{-\alpha} +O(t^{-\alpha+\zeta})$  for some
  $C_0>0$ and $\zeta>0$. Let $\varepsilon_0>0$. Recall $\varphi_n$ given
  by \reff{eq:fn}. 
Then we have, for $n\geq 2$, 
\begin{equation}
   \label{eq:phin}
\phi_n(u)=1 -\frac{u}{\gamma} + \frac{u^\alpha}{\gamma} + R(n,u) ,
\end{equation}
where $\displaystyle R(n,u)= \left(u
  \varphi_n + u^2 \right)h(n,u)$ with $\sup_{u\in [0,K], n\geq 2}
|h(n,u)|<\infty $.  
\end{lem}
\begin{proof}
We have 
\begin{multline*}
\phi_n(u)\\
\begin{aligned}
   &=\E\left[\expp{-uX_1^{(n)}}\right]
=\sum_{k=1}^{n-1}\expp{-uk}\P(X_1^{(n)}=k)\\
&=\sum_{k=1}^{n-1}\expp{-uk}\P(X_1^{(n)}\ge k) -
\sum_{k=2}^{n}\expp{-u(k-1)}\P(X_1^{(n)}\ge k)\\ 
&=\expp{-u} + \sum_{k=2}^{n-1}\expp{-uk}(1-\expp{u})\P(X_1^{(n)}\ge k)\\ 
&=\expp{-u} +
(1-\expp{u})\sum_{k=2}^{n-1}\frac{\expp{-uk}}{g_n}
\int_0^1\frac{n!}{k!(n-k-1)!}t^k(1-t)^{n-k-1}\rho(t)\;  
dt\\  
&=\expp{-u} +
(1-\expp{u})\frac{n}{g_n}\int_0^1\!\!\left[(1-t(1-\expp{-u}))^{n-1} \!\!-
  (1-t)^{n-1} \!\!- (n-1)\expp{-u}t(1-t)^{n-2}\right]\rho(t)dt\\ 
&=1 + (1-\expp{u})\frac{n}{g_n}\int_0^1\left[(1-t(1-\expp{-u}))^{n-1} -
  (1-t)^{n-1} \right]\rho(t)\; dt ,
\end{aligned}
\end{multline*}
where we used \reff{eq:gn} for the last equality. 
Using \reff{eq:E[Xn1]prime}, this implies 
\begin{equation}
   \label{eq:phinDL}
\phi_n(u)=
1 + (1-\expp{u})\frac{n}{g_n}A + (1-\expp{u})\E[X^{(n)}_1].
\end{equation}
with $\displaystyle  A=\int_0^1\left[(1-t(1-\expp{-u}))^{n-1} -1
 \right]\rho(t)\; dt  $. 

Thanks to Lemma \ref{lem:M1}, we have that 
\begin{equation}
   \label{eq:h1}
(1-\expp{u})\E[X^{(n)}_1]= -\frac{u}{\gamma} + \left(u^2+
  u\varphi_n\right) h_1(n,u),
\end{equation}
where $\sup_{u\in [0,K], n\geq 2} |h_1(n,u)|<\infty $.

To compute $A$, we set $a=(1-\expp{-u})$ and
$f(t)=t^{-\max(\alpha-1-\zeta,0)} + t^{-\varepsilon_0}
\ind_{\{\alpha-\zeta=1\}} 
$.  An integration by part gives 
\begin{align*}
   A
&=-a(n-1)\int_0^1(1-at)^{n-2}\left(\int_t^1 \rho(r) \;
  dr\right)\; dt\\ 
&=- a (n-1) C_0\int_0^1 (1-at)^{n-2}\left(\frac{t^{1-\alpha}}{\gamma}
  +O(f(t)) \right)\; dt\\ 
&=-A_1+A_2,
\end{align*}
with $\displaystyle 
A_1=
 \frac{a (n-1)}{\gamma} C_0 \int_0^1 (1-at)^{n-2}t^{1-\alpha} \; dt$
and $\displaystyle 
A_2=  a(n-1) \int_0^1 (1-at)^{n-2} O(f(t))\; dt $.
We have 
\begin{align*}
   A_1
&= \frac{a^{\alpha-1}  (n-1)}{\gamma}  C_0\int_0^a
(1-t)^{n-2}t^{1-\alpha} \; dt \\
&=\frac{a^{\alpha-1} (n-1)}{\gamma} C_0 
\int_0^1
(1-t)^{n-2}t^{1-\alpha} \; dt - 
 \frac{a^{\alpha-1}
   (n-1)}{\gamma}  C_0\int_a^1 (1-t)^{n-2}t^{1-\alpha} \; dt \\
&=\frac{a^{\alpha-1} (n-1)}{\gamma} C_0
\frac{\Gamma(n-1)\Gamma(2-\alpha)}{\Gamma(n+1-\alpha)}  - 
 \frac{a^{\alpha-1}
   (n-1)}{\gamma}  C_0\int_a^1 (1-t)^{n-2}t^{1-\alpha} \; dt 
\end{align*}
Since $a\geq 0$, we have for $u\in [0,K]$ and $n\geq 2$
\[
0\leq \frac{a^{\alpha-1}
  (n-1)}{\gamma}  \int_a^1 (1-t)^{n-2}t^{1-\alpha} \; dt 
\leq \frac{
  (n-1)}{\gamma}  \int_a^1 (1-t)^{n-2} \; dt \leq \inv{\gamma}.
\]
Using \reff{eq:expansionGamma} and Lemma \ref{lem:dlg}, we get 
 $\displaystyle |A_1-\frac{a^{\alpha-1} }{\gamma}
\frac{g_n}{n}|\leq c (1+ n^{\alpha-1-\min(\zeta,1)})
\leq cn^{\max(\alpha-1-\zeta,0)}
$, where 
$c$ does not  depend on $n$ and $u\geq 0$. 
We also have, using \reff{eq:gamma-rel} and \reff{eq:expansionGamma}
\[
|A_2|
\leq c a(n-1) \int_0^1 (1-at)^{n-2} f(t)\; dt 
 \leq  c (n^{\max(\alpha-1-\zeta,0)}+ n^{\varepsilon_0}
\ind_{\{\alpha-\zeta=1\}}).
\]
We deduce, using  Lemma \ref{lem:dlg} twice,  that 
\[
|A + \frac{a^{\alpha-1} }{\gamma}
\frac{g_n}{n}|\leq c(n^{\max(\alpha-1-\zeta,0)}+ n^{\varepsilon_0}
\ind_{\{\alpha-\zeta=1\}})\leq c \frac{g_n}{n}\varphi_n .
\]
We deduce that 
\begin{equation}
   \label{eq:h2}
(1-\expp{u}) \frac{n}{g_n}A = (1-\expp{u}) \left(-
  \frac{(1-\expp{-u})^{\alpha-1}}{\gamma} +\varphi_n  O(1)\right)
= \frac{u^\alpha}{\gamma} +\left(u^{\alpha+1}+
  u\varphi_n \right)h_2(n,u),
\end{equation}
where $\sup_{u\in [0,K], n\geq 2} |h_2(n,u)|<\infty $. 
Then use the expression of $\phi_n$ given by  \reff{eq:phinDL} as well
as \reff{eq:h1} and \reff{eq:h2} to end the proof. 
\end{proof}

\section{Asymptotics for the number of jumps}
\label{sec:taun}
Let $\alpha\in (1,2)$. We  assume   that $\rho(t)=C_0t^{-\alpha}
+O(t^{-\alpha+\zeta})$  for some 
  $C_0>0$ and $\zeta>1-1/\alpha$.

Let $V=(V_t,t\geq 0)$ be a $\alpha$-stable L\'evy process with no positive
jumps   (see  chap.    VII   in  \cite{b:pl})   with  Laplace   exponent
$\psi(u)=u^\alpha/\gamma$: for all $u\geq 0$, $\E[\expp{-uV_t}]=\expp{t
  u^\alpha/\gamma}$.

Lemma  \ref{lem:cvloiXn} implies  that $(X_1^{(n)},  \ldots, X^{(n)}_k)$
converges in distribution to $(X_1, \ldots, X_k)$ where $(X_k, k\geq 1)$
is a sequence of independent  random variables distributed as $X$. 
Using Lemma~\ref{lem:cvloiXn} and \reff{eq:expansionGamma}, we get that
${\mathbb P}(X\geq k)\sim_{+\infty
}\inv{\Gamma(2-{\alpha})}k^{-{\alpha}}$. Hence 
Proposition 9.39 in \cite{b:p} implies that the law of $X$ is in the domain of  attraction of the ${\alpha}$-stable distribution. We set 
$\displaystyle  W^{(n)}_t=n^{-1/\alpha} \sum_{k=1}^{\lfloor  nt \rfloor}(X_k-\inv{\gamma})$
for $t\in[0,{\gamma}]$. An easy calculation using the Laplace transform of $X$ shows that for fixed $t$ the sequence $W^{(n)}_t$ converges in distribution to $V_t$. Then using Theorem 16.14 in \cite{k:fmp}, we get that the process  $(W^{(n)}_t,  t\in[0,\gamma])$ converges in distribution to $V=(V_t,   t\in[0,\gamma])$.
We shall give in Corollary \ref{cor:cvVn} a similar result with
$X_k$ replaced by $X^{(n)}_k$. 

We  first  give  a  proof  of  the convergence  of  $\tau_n$,  see  also
\cite{gy:ncl} and \cite{dimr:rrc} for  a different proof.   We will use
that $\displaystyle 
\sum_{i=1}^{\tau_n      }(X^{(n)}_i-\inv{\gamma})=n-1      -\frac{\tau_n
}{\gamma}$.

\begin{prop}\label{prop:tau_n}
  We assume  that $\zeta>1-1/\alpha$. We have  the following convergence
  is distribution
\[
n^{-\inv{\alpha}}\left(n-\frac{\tau_n
  }{\gamma}\right)
\; \xrightarrow[n\rightarrow \infty ]{\text{(d)}} \; 
V_\gamma. 
\]
\end{prop}

\begin{proof}
  Using  \cite{mrs:nmgf},  it is  enough  to  prove that  $\displaystyle
  \lim_{n\rightarrow                       \infty                      }
  \E[\expp{-un^{-\inv{\alpha}}\left(n-\frac{\tau_n
      }{\gamma}\right)}]=\expp{u^\alpha}$   for  all  $u\geq   0$.   Let
  $\cy=(\cy_k,  k\geq 0)$  be the  filtration generated  by  $Y$. Notice
  $\tau_n$ is an $\cy$-stopping time.  For fixed $n$, and for any $v\geq
  0$, the process $(M_{v,k},k\geq 0)$ defined by
\[
M_{v,k}=\prod_{i=1}^{k}  \left(\exp{-v X^{(n)}_i
  -\log\phi_{Y^{(n)}_{i-1}}(v)}  \right)
\]
is a bounded martingale w.r.t. the filtration $\cy$. Notice that
$\E[M_{v,k}]=1$. As $X_i=0$ for
$i>\tau_n$,  we also have 
\begin{equation}\label{eq:def of M_{v,k}}
M_{v,k}=\prod_{i=1}^{k\wedge\tau_n } \left(\exp{-v X^{(n)}_i
    -\log\phi_{Y^{(n)}_{i-1}}(v)} \right).
\end{equation}

Let $u\geq 0$ and consider a non-negative sequence $(a_n, n\geq 1)$
which converges to $0$. Using \reff{eq:phin}, we get that :
\begin{align*}
M_{ua_n,k}
&=\exp\left(-ua_n\sum_{i=1}^{k\wedge\tau_n }\ixen_i
-\sum_{i=1}^{k\wedge\tau_n }\left( -\frac{ua_n}{\gamma} + \frac{u^\alpha
    a_n^\alpha}{\gamma} +R(\igreken_{i-1},ua_n)\right)\right). 
\end{align*}
In particular, we have 
\begin{equation}\label{eq:Mun,taun}
M_{ua_n,\tau_n } = 
\exp{\left(- u a_n(n-1 - \frac{\tau_n }{\gamma})
-\frac{u^\alpha \tau_n  a_n^\alpha}{\gamma}
-\sum_{i=1}^{\tau_n } R(Y_{i-1}^{(n)}, ua_n)\right) }.
\end{equation}

We first give an upper bound for $\sum_{i=1}^{\tau_n } R(Y_{i-1}^{(n)},
ua_n)$. 
\begin{lem}\label{lm:maj_rest}
We assume that $\zeta>1-1/\alpha$.    Let $K>0$. Let ${\eta}\geq \inv{\alpha}$. There exist ${\varepsilon_1}>0$ and
   $C_{\ref{eq:upperboundR}}(K)$ a finite constant  such that for all
    $n\geq 1$ and $u\in [0,K]$, a.s. with $a_n=n^{-\eta}$, 
\begin{equation}\label{eq:upperboundR}
\sum_{i=1}^{\tau_n } \left|R(Y_{i-1}^{(n)}, ua_n)\right|\leq
C_{\ref{eq:upperboundR}}(K)n^{-\varepsilon_1}. 
\end{equation}
\end{lem}

\begin{proof}
  Notice that $\tau_n\leq n-1$. We have seen in Lemma~\ref{lm:phin} that
  $\displaystyle  R(n,u)= \left(u  \varphi_n +  u^2  \right)h(n,u)$ with
  $\bar{h}(K)=\sup_{u\in   [0,K],  n\geq   2}   |h(n,u)|<\infty  $   and
  $\varphi_n$ given by \reff{eq:fn}. We have $\displaystyle
  2-\alpha-\inv{\alpha}=-\alpha(1-1/\alpha)^2<0$. As  $\varepsilon_0>0$ is
  arbitrary in \reff{eq:fn}, we can take $\varepsilon_0$ small enough so that
  $1-\alpha+\varepsilon_0<0$ and $2-\alpha+\varepsilon_0-1/\alpha<0$. We
  have  
\[
a_n\sum_{i=1}^{\tau_n }\varphi_{{Y_{i-1 }^{(n)}}} \leq
  n^{-1/\alpha} \sum_{j=1}^n \varphi_j \leq c\begin{cases}
   {n^{1-\zeta-\inv{\alpha}} } &\quad \text{if}\quad \zeta<\alpha-1,\\
 {n^{2-\alpha+\varepsilon_0-\inv{\alpha}}} &\quad \text{if}\quad
 \zeta=\alpha-1,\\ 
 n^{2-\alpha-\inv{\alpha}} &\quad \text{if}\quad \zeta>\alpha-1.
  \end{cases}
\]
For   $\varepsilon_1>0$   less   than   the  two   positive   quantities
$\displaystyle       -1+\zeta+\inv{\alpha}$      and      $\displaystyle
-2+\alpha-\varepsilon_0+\inv{\alpha}    $,   we    have   $\displaystyle
a_n\sum_{i=1}^{\tau_n     }\varphi_{{Y_{i-1     }^{(n)}}}     \leq     c
n^{-\varepsilon_1}$.  We deduce that, for $u\in [0,K]$,
\begin{align*}
\sum_{i=1}^{\tau_n } \left|R(Y_{i-1}^{(n)}, ua_n)\right| 
&\leq 
\bar{h}(K)\sum_{i=1}^{\tau_n }\left(\varphi_{Y_{i-1 }^{(n)}}
  ua_n+(ua_n)^2\right)\\ 
&\leq \bar{h}(K)\sum_{j=1}^n\left(\varphi_j Ka_n+(Ka_n)^2\right)\\
&\leq  c \bar{h}(K) (Kn^{-\varepsilon_1}+K^2n^{1-\frac{2}{\alpha}}),
\end{align*}
for  some  constant  $c$  independent  of  $n$,  $u$  and  $K$.   Taking
$\varepsilon_1>0$ small enough    so     that     $\displaystyle
\varepsilon_1<\frac{2}{\alpha}-1$, we then get \reff{eq:upperboundR}.
\end{proof}

Next we prove the following Lemma.
\begin{lem}\label{lm:1st_bound_for_tau_n}
We assume that $\zeta>1-1/\alpha$. Let  ${\varepsilon}>0$. The sequence 
$(n^{-(1/{\alpha})-{\varepsilon}}(n-1-\frac{\tau_n }{\gamma}), n\geq 1) $
converges in probability to $0$. 
\end{lem}

\begin{proof}
We set $a_n=n^{-\inv{\alpha}-\varepsilon}$. 
Notice that 
\[
\expp{- u a_n ( n-1- \frac{\tau_n}{\gamma})} 
=M_{ua_n, \tau_n }\expp{\frac{u^\alpha \tau_n a_n}{\gamma} +
  \sum_{i=1}^{\tau_n } 
  R(Y_{i-1}^{(n)}, ua_n)}. 
\]
As $\tau_n \leq n-1$, we have $0\leq \tau_n  a_n^\alpha\leq
n^{-\alpha\varepsilon}$. Using  \eqref{eq:upperboundR}, we get for
$u\geq 0$ 
\[
\E[M_{ua_n,\tau_n }]\expp{-C_{\ref{eq:upperboundR}} (u) n^{-\varepsilon_1}}
\leq\E[\expp{-ua_n(n-1-\frac{\tau_n }{{\gamma}})}]
\leq\E[M_{ua_n,\tau_n }]\expp{C_{\ref{eq:upperboundR}}(u)n^{-\varepsilon_1}
  +\frac{u^\alpha   n^{-\alpha \varepsilon} }{\gamma}}.
\]
As $\tau_n$ is bounded, the stopping time theorem gives
$\E[M_{ua_n,\tau_n }]=1$. 
We deduce that,   for all $u\geq 0$,
$\displaystyle  \lim_{n\rightarrow \infty }
\E[\expp{-ua_n(n-1-\frac{\tau_n }{{\gamma}})}]=1$. Using
\cite{mrs:nmgf}, we get the convergence in law of
$a_n(n-1-\frac{\tau_n }{{\gamma}})$ to 0, and then in probability as
the limit is constant. 
\end{proof}

Let $a_n = n^{-\inv{\alpha}}$ and $u\geq 0$. 
We have 
\begin{multline}
\label{eq:diff}
\E\left[\expp{-ua_n(n-1-\frac{\tau_n }{{\gamma}})}\right]\\
\begin{aligned}
&=   \E\left[\expp{-ua_n(n-1-\frac{\tau_n }{{\gamma}})}
\left(1-\expp{-u^\alpha a_n^\alpha(\frac{\tau_n
    }{{\gamma}}-n)}\right)\right]+ \E\left[\expp{-ua_n(n-1-\frac{\tau_n }{{\gamma}})}
\expp{-u^\alpha a_n^\alpha(\frac{\tau_n }{{\gamma}}-n)}\right]\\
&=I_1+I_2,   
\end{aligned}   
\end{multline}
with  $\displaystyle I_1=\E\left[\expp{-ua_n(n-1-\frac{\tau_n }{{\gamma}})}
\left(1-\expp{-u^\alpha a_n^\alpha(\frac{\tau_n
    }{{\gamma}}-n)}\right)\right]$ and $\displaystyle
I_2=\E\left[M_{ua_n,\tau_n }\expp{u^\alpha+\sum_{i=1}^{\tau_n }
    R(Y_{i-1}^{(n)}, ua_n)}\right]$.

Using  \eqref{eq:upperboundR} and  $\E[M_{ua_n,\tau_n }]=1$, we get 
\[
\expp{u^\alpha-C_{\ref{eq:upperboundR}}(u)n^{-\varepsilon_1}}
\leq  I_2
\leq \expp{u^\alpha+C_{\ref{eq:upperboundR}}(u)n^{-\varepsilon_1}}.
\]
This implies that $\displaystyle  \lim_{n\rightarrow \infty }
I_2=\expp{u^\alpha}$. 

We now prove that $\displaystyle \lim_{n\rightarrow \infty } I_1=0$. 
Recall that  $\tau_n \leq n-1$ so that $\tau_n  a_n^\alpha\leq 1$ and
thanks to \eqref{eq:upperboundR}, we get 
\[
\E[\expp{-ua_n(n-1-\frac{\tau_n }{{\gamma}})}]
=\E\left[M_{ua_n,\tau_n }\expp{\frac{u^\alpha \tau_n
      a_n^\alpha}{\gamma}+ \sum_{i=1}^{\tau_n } R(Y_{i-1}^{(n)},
    ua_n)}\right] 
\leq M(u) \E[M_{ua_n,\tau_n }]
= M(u), 
\]
where $M(u)$ is a constant which does not  depend on $n$. By Cauchy-Schwarz' inequality, we get that 
\begin{align*}
I_1=\E\left[\expp{-ua_n(n-1-\frac{\tau_n }{{\gamma}})}
\left(1-\expp{-u^\alpha a_n^\alpha(\frac{\tau_n }{{\gamma}}-n)}\right)\right]^2
&\leq \E\left[\expp{-2ua_n(n-1-\frac{\tau_n }{{\gamma}})}\right]
\E\left[\left(1-\expp{-u^\alpha a_n^\alpha(\frac{\tau_n
      }{{\gamma}}-n)}\right)^2\right]\\ 
&\leq M(2u)
\E\left[\left(1-\expp{-u^\alpha \inv{n}(\frac{\tau_n
      }{{\gamma}}-n)}\right)^2\right]. 
\end{align*} 
Notice $(\inv{n}(\frac{\tau_n }{{\gamma}}-n), n\geq 1)$ is bounded from
below and above by finite constants, and thanks to Lemma
\ref{lm:1st_bound_for_tau_n} it converges to $0$ in probability. Hence,
we deduce that 
\[
\lim_{n\rightarrow \infty } \E\left[\left(1-\expp{-u^\alpha
      \inv{n}(\frac{\tau_n }{{\gamma}}-n)}\right)^2\right]= 0. 
\]
This implies that $\displaystyle \lim_{n\rightarrow \infty } I_1=0$. 

{From} the convergence of $I_1$ and $I_2$, we deduce from \reff{eq:diff}
that $\displaystyle  \lim_{n\rightarrow \infty }
\E\left[\expp{-ua_n(n-1-\frac{\tau_n }{{\gamma}})}\right]
=\expp{u^\alpha}$. 
This ends the proof of the Proposition. 
\end{proof}

We now  give  a  general result. 
\begin{prop}
\label{prop:cvV}
We assume that $\zeta>1-1/\alpha$. Let $f_n:{\mathbb R}_+\to {\mathbb R}_+$ be uniformly 
bounded  functions such that 
\begin{equation*}
 \kappa  =\lim_{n\to\infty }\inv{n}\sum_{k=1}^{\lfloor  n\gamma
  \rfloor}f_n(k/n)^\alpha 
\end{equation*}
exists.  Then we have the following convergence in distribution
\begin{equation}
\label{eq:cvVnfn}
V^{(n)}(f_n):=n^{-\inv{\alpha}}\sum_{k=1}^{{\tau}_n}f_n(k/n)
(X_k^n-\inv{\gamma}) \; \xrightarrow[n\rightarrow \infty ]{\text{(d)}} \; 
 \kappa ^{1/\alpha}  V_1.
\end{equation}
In particular, if  $f:{\mathbb R}_+\to {\mathbb R}_+$ is 
a bounded locally Riemann integrable function, then 
\begin{equation}\label{eq:cvVnf}
V^{(n)}(f)=n^{-\inv{\alpha}}\sum_{k=1}^{{\tau}_n}f(k/n)
(X_k^n-\inv{\gamma}) 
\; \xrightarrow[n\rightarrow \infty ]{\text{(d)}} \; 
 \int_0^\gamma f(t) dV_t,
\end{equation}
where the distribution of $\int_0^\gamma f(t) dV_t$ is characterized by
its Laplace transform: for $u\geq 0$, 
\begin{equation}
   \label{eq:TLfdVt}
{\mathbb E}[\exp(-u\int_0^\gamma f(t)
dV_t)]=\exp\left(\frac {u^{\alpha}}{\gamma}
  \int_0^{\gamma}f^{\alpha}(t)\,dt\right)  .
\end{equation}
\end{prop}

If we apply this Proposition with step functions, we deduce
the following result. 
\begin{cor}\label{cor:cvVn}
We assume that $\zeta>1-1/\alpha$. Let 
$V^{(n)}_t=V^{(n)}(\ind_{[0,t]})=n^{-1/{\alpha}}\sum_{k=1}^{\lfloor nt\rfloor\wedge\tau_n}(X^{(n)}_k-\inv{\gamma})$
for $t\in [0,{\gamma})$, and $V^{(n)}_\gamma= V^{(n)}(\ind) = n^{-1/\alpha}
\left(n-1-\frac{\tau_n}{\gamma} \right)$. 
The finite-dimensional marginals of the process $(V^{(n)}_t, t\in
[0,\gamma])$ converges in law to those of the process $(V_t,t\in
[0,{\gamma}])$. 
\end{cor}

\begin{proof}
Thanks to \cite{mrs:nmgf},  it is enough to prove that 
\[
{\mathbb E}[\exp(-uV^{(n)}(f_n))]\; \xrightarrow[n\rightarrow \infty
]{} \;   \expp { \kappa   u^\alpha/\gamma}.
\]
Taking $uf_n$ as $f_n$, we shall only consider the case $u=1$. 

We set $a=\sup_{n\geq 1,x\geq 0}|f_n(x)|$ and for any bounded function
$g$, 
\begin{equation*}
A_n(g)
=
\exp{
\sum_{k=1}^{{\tau}_n}
\left(
-n^{-1/{\alpha}}g(k/n)X^{(n)}_k
-\log\phi_{Y^{(n)}_{k-1}}(n^{-\inv{\alpha}}g(k/n))
\right)
}.
\end{equation*}
A martingale argument provides that ${\mathbb E}[A_n(g)]=1$. Using
\reff{eq:phin}, we get that : 
\begin{align*}
A_n(g)
&=
\exp
{\left(
-n^{-1/{\alpha}}\sum_{k=1}^{{\tau}_n}g(k/n)(X^{(n)}_k-\inv{\gamma})
-n^{-1}\sum_{k=1}^{{\tau}_n}\frac{g^{\alpha}(k/n)}{{\gamma}}
-\sum_{k=1}^{{\tau}_n}R(Y^{(n)}_{k-1},n^{-\inv{\alpha}}g(k/n))
\right)
}\\
&=
\exp
{\left(
-V^{(n)}(g)
-n^{-1}\sum_{k=1}^{{\tau}_n}\frac{g^{\alpha}(k/n)}{{\gamma}}
-\sum_{k=1}^{{\tau}_n}R(Y^{(n)}_{k-1},n^{-\inv{\alpha}}g(k/n))
\right)
}.
\end{align*}
Let $\displaystyle \Lambda_n= n^{-1}\sum_{k=1}^{\lfloor
  n{\gamma}\rfloor}\frac{f_n^{\alpha}(k/n)}{{\gamma}} 
-
n^{-1}\sum_{k=1}^{{\tau}_n}\frac{f_n^{\alpha}(k/n)}{{\gamma}}
$ 
and write 
 \[
{\mathbb E}\left[\expp{-V^{(n)}(f_n)}\right]=I_1+I_2
\]
with 
$\displaystyle 
I_1
={\mathbb E}\left[\expp{-V^{(n)}(f_n)}
\left(
1-
\expp{
\Lambda_n
}
\right)
\right]
$ and 
$ 
I_2=
{\mathbb E}\left[\expp{-V^{(n)}(f_n)}
\expp{
\Lambda_n
}
\right]
$.

First of all, let  us prove that $I_1$ converges to 0  when $n$ tends to
$\infty $.   Recall that the  functions $f_n$ are uniformly  bounded by
$a$. Thanks to \eqref{eq:upperboundR}, we have 
\[
\E[\expp{-2V^{(n)}(f_n) }]
=\E[\expp{-V^{(n)}(2f_n) }]
=\E\left[A_n(2f_n) 
\expp{n^{-1}\sum_{k=1}^{{\tau}_n}\frac{2^\alpha f_n^{\alpha}(k/n)}{{\gamma}}
+\sum_{k=1}^{{\tau}_n}R(Y^{(n)}_{k-1},n^{-\inv{\alpha}}2f_n(k))}
\right]
\leq M,
\]
where $M$ is a finite constant which does not  depend on $n$. By Cauchy-Schwarz' inequality, we get that 
\[
(I_1)^2
\leq 
\left({\mathbb E}\left[\expp{-V^{(n)}(f_n)}
\left|1-\expp{\Lambda_n} \right|\right]\right)^2
\leq
 {\mathbb E}\left[\expp{-V^{(n)}(2f_n)}\right]
{\mathbb E}\left[\left(1-\expp{\Lambda_n}\right)^2\right]
\leq 
M{\mathbb E}\left[\left(1-\expp{\Lambda_n}\right)^2
\right].
\]
Moreover as $|1-\expp{x}|\leq \expp{|x|}-1$ and $\displaystyle
\Lambda_n \leq   \frac{a^\alpha}{n\gamma } |\lfloor n\gamma\rfloor-
\tau_n|$, we get  
\begin{equation}
\label{eq:maj(1-elambda)2}
{\mathbb E}
\left[\left(1-\expp{\Lambda_n}\right)^2
\right]
\leq 
{\mathbb E}
\left[
\left(1-
\expp{
\frac{|\lfloor n{\gamma}\rfloor -{\tau}_n|a^{\alpha}}{n{\gamma}}
}
\right)^2
\right].
\end{equation}
The quantity $\displaystyle  \frac{|\lfloor n{\gamma}\rfloor
  -{\tau}_n|a^{\alpha}}{n{\gamma}}$ is bounded and goes to $0$ in
probability when $n$ goes to infinity. Therefore, the right-hand side of \reff{eq:maj(1-elambda)2} converges to 0. This implies that
$\lim_{n\rightarrow\infty } I_1=0$. 

Let us now consider the convergence of  $I_2$. Remark that 
\begin{equation*}
I_2={\mathbb E}
\left[
A_n(f_n) 
\expp{
n^{-1}\sum_{k=1}^{\lfloor n{\gamma} \rfloor }\frac{f_n^{\alpha}(k/n)}{{\gamma}}
+\sum_{k=1}^{\tau_n} R(Y^{(n)}_{k-1},n^{-\inv{\alpha}}f_n(k))}
\right].
\end{equation*}
Recall that $f_n$ is bounded by 
$a$  and that ${\mathbb E}[A_n(f_n)]=1$. Using Lemma~\ref{lm:maj_rest},
we get for some $\varepsilon>0$  
\begin{multline}
\label{eq:maj2}
\expp{-C_{\ref{eq:upperboundR}}(a )n^{-{\varepsilon_1}}
 -n^{-1}\sum_{k=1}^{\lfloor n{\gamma}
   \rfloor}\frac{f_n^{\alpha}(k/n)}{{\gamma}}}\\ 
\leq {\mathbb E}
\left[
A_n(f_n)
\expp{
n^{-1}\sum_{k=1}^{\lfloor n{\gamma}
   \rfloor}\frac{f_n^{\alpha}(k/n)}{{\gamma}}
+\sum_{k=1}^{\tau_n} R(Y^{(n)}_{k-1},n^{-\inv{\alpha}}f_n(k))}
\right]\\
\leq \expp{C_{\ref{eq:upperboundR}}(a) n^{-\varepsilon_1} +
n^{-1}\sum_{k=1}^{\lfloor n{\gamma}
   \rfloor}\frac{f_n^{\alpha}(k/n)}{{\gamma}}}.
\end{multline}
As 
$\displaystyle \lim_{n\to\infty
}\inv{n}\sum_{k=1}^{\lfloor n{\gamma}\rfloor}f_n^{\alpha}(k/n)
= \kappa $, we get that $\lim_{n\to \infty }I_2=\expp{ \kappa 
  /\gamma}$, which achieves the proof of \reff{eq:cvVnfn}.  To get
\reff{eq:cvVnf}, notice that 
$\displaystyle 
\kappa =\lim_{n\to\infty }\inv{n}\sum_{k=1}^{\lfloor  n\gamma
  \rfloor}f(k/n)^\alpha 
=\int_0^\gamma f(t)^\alpha \; dt $.
\end{proof}

\section{First approximation of the length of the coalescent
  tree}\label{sec:prelim} 
Let $\alpha\in (1,2)$. We  assume   that $\rho(t)=C_0t^{-\alpha} +O(t^{-\alpha+\zeta})$  for some
  $C_0>0$ and $\zeta>1-1/\alpha$.

Recall that the length of the coalescent tree up to the
$\lfloor nt \rfloor$-th coalescence is, for $t\geq 0$,  given by
\reff{eq:Lnt}. 
The next Lemma  gives an upper bound on the error  when one replaces the
exponential random variables by their mean.
\begin{lem}
\label{lem:Ltilde}
   For $t\geq 0$, let 
\[
\tilde L^{(n)}_{t}=\sum_{k=0}^{\lfloor nt \rfloor \wedge (\tau_n-1)}
\frac{Y^{(n)}_{k} }{g_{Y^{(n)}_{k} }} .
 \]
There exists a finite constant $C_{\ref{eq:DLexp}}$ such that for all
$t\geq 0$, we have 
\begin{equation}
   \label{eq:DLexp}
\E\left[(L^{(n)}_t -\tilde L^{(n)}_{t})^2 \right] \leq
C_{\ref{eq:DLexp}} \begin{cases}
    n^{3-2\alpha} &\text{if }\alpha<3/2,\\
  \log(n)  &\text{if }\alpha=3/2,\\
   1 &\text{if }\alpha>3/2.
 \end{cases} 
\end{equation}
\end{lem}

\begin{proof}
   Conditionally on $\cy$, the random variables $\displaystyle
   \frac{Y^{(n)}_{k} 
   }{g_{Y^{(n)}_{k} }} (E_k-1)$ are independent with zero mean. We
   deduce that 
\begin{align*}
    \E\left[(L^{(n)}_t -\tilde L^{(n)}_{t})^2 |\cy\right]
 &=
 \E\left[\left(\sum_{k=0}^{\lfloor nt \rfloor \wedge (\tau_n-1)}
 \frac{Y^{(n)}_{k} }{g_{Y^{(n)}_{k} }} (E_k-1)\right)^2 |\cy\right]\\
&=
\sum_{k=0}^{\lfloor nt \rfloor \wedge (\tau_n-1)} \left(\frac{Y^{(n)}_{k}
}{g_{Y^{(n)}_{k} }}\right)^2 \\
&\leq \sum_{\ell=1}^{n}  \left(\frac{ \ell}{g_\ell}\right)^2. 
\end{align*}
Thanks to \reff{eq:dlg}, we get 
\[
 \E\left[(L^{n}_t -\tilde L^{(n)}_{t})^2 |\cy\right] \leq c
 \sum_{\ell=1}^{n} \ell^{2-2\alpha} \leq  c \begin{cases}
    n^{3-2\alpha} &\text{if }\alpha<3/2,\\
  \log(n)  &\text{if }\alpha=3/2,\\
    1&\text{if }\alpha>3/2,
\end{cases}
\]
where $c$ is non random. This implies the result.
\end{proof}

\begin{lem}
\label{lem:Lhat}
   For $t\geq 0$, let 
\[
\hat L^{(n)}_{t}= \sum_{k=0}^{\lfloor nt \rfloor \wedge (\tau_n-1)}
\left(Y^{(n)}_{k}\right)^{-\gamma}.
 \]
There exists a finite constant $C_{\ref{eq:hDLn}} $ such that for all $t\geq 0$, we have 
\begin{equation}
   \label{eq:hDLn}
|\tilde L^{(n)}_t -\frac{\hat L^{(n)}_{t}}{C_0\Gamma(2-\alpha)}| \leq
C_{\ref{eq:hDLn}}\; 
\begin{cases}
   n^{2-\alpha-\zeta} &\text{if } \zeta<2-\alpha,\\
\log(n) &\text{if } \zeta=2-\alpha,\\
  1 &\text{if } \zeta>2-\alpha.
\end{cases}
\end{equation}
\end{lem}

\begin{proof}
Use \reff{eq:dlg} to get that 
\[
   \tilde L^{(n)}_t -\frac{\hat L^{(n)}_{t}}{C_0\Gamma(2-\alpha)} 
=\sum_{k=0}^{\lfloor nt \rfloor \wedge (\tau_n-1)}
\left(Y^{(n)}_{k}\right)^{-\gamma} O
\left(\left(Y^{(n)}_{k}\right)^{-\min(\zeta,1)} \right).
\]
We deduce that 
\[
| \tilde L^{(n)}_t -\frac{\hat L^{(n)}_{t}}{C_0\Gamma(2-\alpha)}| \leq
c \sum_{\ell=1}^n 
\ell^{-\alpha+1 - \min(\zeta,1)}
\leq  c 
\begin{cases}
   n^{2-\alpha-\zeta} &\text{if } \zeta<2-\alpha,\\
\log(n) &\text{if } \zeta=2-\alpha,\\
  1 &\text{if } \zeta>2-\alpha.
\end{cases}
\]
\end{proof}

\section{Limit distribution of $\hat L_t^{(n)}$}
\label{sec:hatLn}
Let  $\alpha\in (1,2)$ and $\gamma=\alpha-1$. For $t\in [0,\gamma]$, we set 
\[
v(t) =\int_0^t
    \left(1-\frac{r}{\gamma}\right)^{-\gamma} dr.
\] 
\begin{theo}\label{theo:main}
We  assume   that $\rho(t)=C_0t^{-\alpha}
+O(t^{-\alpha+\zeta})$  for some 
  $C_0>0$ and $\zeta>1-1/\alpha$.
Then for all $t\in (0,\gamma)$, we have    that 
\begin{enumerate}
   \item The following convergence in probability holds:
\begin{equation}
   \label{eq:CVLNPv}
n^{-2+\alpha} \hat L^{(n)}_t \; \xrightarrow[n\rightarrow \infty ]{\P}
\;  v(t). 
\end{equation}
   \item  The following convergence in distribution  holds:
\begin{equation}
   \label{eq:CVLNDV}
n^{-1+\alpha -1/\alpha} (  \hat L^{(n)}_t-n^{2-\alpha}
v(t))  \; \xrightarrow[n\rightarrow \infty ]{\text{(d)}}
\; (\alpha-1) \int_0^t dr\; (1-\frac{r}{\gamma})^{-\alpha} V_r.
\end{equation}
\end{enumerate}
\end{theo}

\begin{proof}[Proof of Theorem \ref{theo:main}]
Let $\varepsilon_2\in (0,\gamma)$ and  $t\in (0,\gamma-\varepsilon_2)$.
We use a Taylor expansion to get 
\begin{align}
\nonumber
\hat L^{(n)}_t
&= \sum_{k=0}^{\lfloor nt \rfloor\wedge({\tau}_n-1)}\left(n-\sum_{i=1}^kX^{(n)}_i\right)^{-{\gamma}}\\
\nonumber
&= \sum_{k=0}^{\lfloor nt \rfloor\wedge(\tau_n-1) }\left(n-\frac{k}{{\gamma}}-\sum_{i=1}^k(X^{(n)}_i-\frac{1}{{\gamma}})\right)^{-{\gamma}}\\
\nonumber
&= \sum_{k=0}^{\lfloor nt \rfloor\wedge(\tau_n-1) }\left(n-\frac{k}{{\gamma}}\right)^{-{\gamma}}
\left(1-\Delta_{n,k}\right)^{-{\gamma}}\\
&=I_n+\gamma J_n+ \gamma(\gamma+1) R_n
\label{eq:L=IJR}
\end{align}
with $\displaystyle
\Delta_{n,k}=\frac{\sum_{i=1}^k(X^{(n)}_i- \frac{1}{{\gamma}})}{n-k/{\gamma}}$
and 
\begin{align*}
I_n&=\sum_{k=0}^{\lfloor nt \rfloor\wedge(\tau_n-1) }\left(n-\frac{k}{{\gamma}}\right)^{-{\gamma}},\\
J_n&=\sum_{k=1}^{\lfloor nt \rfloor\wedge(\tau_n-1) }\left(n-\frac{k}{{\gamma}}\right)^{-{\gamma}-1}\sum_{i=1}^k(X^{(n)}_i-\frac{1}{{\gamma}}),\\
R_n&=
\sum_{k=1}^{\lfloor nt \rfloor\wedge(\tau_n-1) }
\left(n-\frac{k}{{\gamma}}\right)^{-{\gamma}}
\int_{0}^{\Delta_{n,k}} 
\left(\Delta_{n,k}-t\right)(1-t)^{-{\gamma}-2}\,dt.
\end{align*}
Notice that a.s. $\Delta_{n,k}<1$,  so that $R_n$ is well defined. 

\noindent
\textbf{Convergence of  $I_n$}.  We first  give an  expansion of  $I_n$ by
considering      $I_n=n^{2-\alpha}      I_{n,1}\ind_{\{nt<\tau_n\}}      +
I_n\ind_{\{nt\geq         \tau_n\}}$         with         $\displaystyle
I_{n,1}=\inv{n}\sum_{k=0}^{\lfloor                nt               \rfloor
}\left(1-\frac{k}{n{\gamma}}\right)^{-{\gamma}}$.   Standard computation
yields
\[
I_{n,1}= v(t)+\inv{n} h_3(n,t),
\]
where   $\displaystyle \sup_{t\in
  (0,\gamma-\varepsilon), n\geq 1} |h_3(n,t)|<\infty$. 
By decomposing according to $\{nt<\tau_n\}$ and $\{nt\geq \tau_n\}$, we
deduce that,  
\[
\P\left(n^{-1+\alpha-1/\alpha} \val{I_n-n^{2-\alpha} v(t)}\geq
  \varepsilon\right) 
 \leq  \P(n^{-1/\alpha}|h_3(n,t)| \geq \varepsilon/2) + \P(nt\geq
 \tau_n). 
\]
According to Lemma \ref{lm:1st_bound_for_tau_n}, $\tau_n/n$ converges in
probability to $\gamma>t$. This implies that 
\begin{equation}
   \label{eq:cvPtn}
 \lim_{n\rightarrow\infty }
\P(nt\geq \tau_n)=0.
\end{equation}
As   $n^{-1/\alpha}|h_3(n,t)|  \leq \varepsilon$  for  $n$ large  enough,
  we 
deduce the following convergence in probability:
\begin{equation}
   \label{eq:cvPI}
n^{-1+\alpha-1/\alpha}  \left(I_n-n^{2-\alpha} v(t)\right)
 \; \xrightarrow[n\rightarrow \infty ]{\P} \; 0. 
\end{equation}


\noindent
\textbf{Convergence of  $J_n$}. To get the convergence of $J_n$, notice that 
\begin{equation}
   \label{eq:Jn}
J_n=\sum_{i=1}^{\lfloor nt \rfloor\wedge(\tau_n-1)
}(X^{(n)}_i-\frac{1}{{\gamma}}) 
\sum_{k=i}^{\lfloor nt \rfloor\wedge(\tau_n-1)}
\left(n-\frac{k}{{\gamma}}\right)^{-\alpha} = n^{1-\alpha} J_{n,1}\ind_{\{nt<\tau_n\}} + J_n\ind_{\{nt\geq
  \tau_n\}}, 
\end{equation}
with  $\displaystyle J_{n,1}= \sum_{i=1}^{\lfloor nt \rfloor\wedge(\tau_n-1)
} f_n(i) (X^{(n)}_i-\frac{1}{{\gamma}})$ and 
$\displaystyle 
f_n(r)= \inv{n} \sum_{j=\lfloor n r \rfloor }^{\lfloor nt \rfloor}
\left(1-\frac{j}{{n\gamma}}\right)^{-\alpha}$. 
The functions $f_n$ are finite and uniformly bounded as for $n\geq
2/\varepsilon_2$, 
\[
0\leq f_n(r)\leq 
f_n(0)=\inv{n} \sum_{k=0 }^{\lfloor nt \rfloor}
\left(1-\frac{k}{{n\gamma}}\right)^{-\alpha} \leq  
\int_0^{\gamma-\varepsilon_2/2}
\left(1-\frac{s}{{\gamma}}\right)^{-\alpha} \; ds <\infty .
\]
Notice that 
\[
\kappa=\lim_{n\to\infty }\inv{n}\sum_{k=1}^{\lfloor n{\gamma}\rfloor}f_n(k)^\alpha =
\int_0^t dr \left( \int_r^t (1- \frac{s}{\gamma} )^{-\alpha} \; ds
\right)^\alpha.
\]


We deduce from Proposition \ref{prop:cvV} that $(n^{-\inv{\alpha}}J_{n,1},
n\geq 2) $ converges in distribution to $\kappa^{1/\alpha} V_1$.
For $\varepsilon'>0$, we have   $\displaystyle \P(\ind_{\{ nt \geq
  \tau_n\}} |J_n|\geq \varepsilon')\leq  \P(nt \geq  \tau_n)$. 
Then we use  \reff{eq:Jn} and \reff{eq:cvPtn}  to conclude that
the following convergence in distribution holds:
\begin{equation}
   \label{eq:cvJNP}
 n^{-1+\alpha-1/\alpha} J_{n}
\; \xrightarrow[n\rightarrow \infty ]{\text{(d)}} \;
\kappa^{1/\alpha} V_1.
\end{equation}

\noindent
\textbf{Convergence of  $R_n$}.
We shall now prove that
$n^{-1+\alpha-1/\alpha} R_n$
converges to $0$ in probability. 
Let $\varepsilon\in (0,\gamma)$. 
We have $\displaystyle R_n= R_{n,1}+ R_{n,2}$, with 
\begin{align*}
R_{n,1}&= \sum_{k=1}^{\lfloor nt
  \rfloor}\left(n-\frac{k}{{\gamma}}\right)^{-{\gamma}}
\ind_{\{k<\tau_n \}} R_{n,1,k},\\
   R_{n,1,k}
&=\ind_{\{\Delta_{n,k}<1-\varepsilon\}} 
\int_{0}^{\Delta_{n,k}} 
\left(\Delta_{n,k}-t\right)(1-t)^{-{\gamma}-2}\,dt,\\
  R_{n,2}
&=\sum_{k=1}^{\lfloor nt
  \rfloor}\left(n-\frac{k}{{\gamma}}\right)^{-{\gamma}} \ind_{\{k<\tau_n \}} \ind_{\{\Delta_{n,k}\geq 1-\varepsilon\}} 
\int_{0}^{\Delta_{n,k}} 
\left(\Delta_{n,k}-t\right)(1-t)^{-{\gamma}-2}\,dt.
\end{align*}

We have for $k\leq n(\gamma-\varepsilon_2)$,
\[
   \E[|R_{n,1,k}|]
\leq c\; \E[(\Delta_{n,k})^2]
\leq \frac{c}{n^2} \; \E\left[\left(\sum_{i=1}^k(X^{(n)}_i-
      \frac{1}{{\gamma}})\right)^2\right] .
\]
Recall   $\cy=(\cy_k,  k\geq 0)$ is the  filtration generated  by  $Y$.
We consider the $\cy$-martingale $N_r=\sum_{j=1}^r \Delta N_r$, with
$\Delta N_r= X^{(n)}_r - \E[X^{(n)}_r|\cy_{r-1}]$. We have 
\[
 \E\left[\left(\sum_{i=1}^k(X^{(n)}_i-
      \frac{1}{{\gamma}})\right)^2\right] 
\leq  2 \E\left[N_k^2 \right]+
2 \E\left[\left(\sum_{i=1}^k(\E[X^{(n)}_i|\cy_{i-1}] -
      \frac{1}{{\gamma}})\right)^2\right] . 
\]
Notice that 
\[
\E\left[N_k^2 \right]
=\E\left[\sum_{i=1}^k (\Delta N_i)^2 \right]
\leq \E\left[\sum_{i=1}^k\E[(X^{(n)}_i)^2|\cy_{i-1}]\right]
\leq \E\left[\sum_{i=1}^{k} (X^{(n)}_i)^2\right].
\]
Using that, conditionally on $\cy_{i-1}$, $X^{(n)}_i$ and $X^{(Y_{i-1})}_1$ have the same distribution, we get that 
\[
\E\left[N_k^2 \right]
\leq \sum_{j=1}^n
\E[(X^{(j)}_1)^2].
\]
Thanks to \reff{eq:M2} and \reff{eq:dlg}, we deduce that 
\[
\E\left[N_k^2 \right] \leq   C_{\ref{eq:M2}} \sum_{j=1}^n
\frac{j^2}{g_j} \leq  c \sum_{j=1}^n j^{2-\alpha} \leq c\; n^{3-\alpha}. 
\]
Using \reff{eq:M1} and \reff{eq:dlg}, we get 
\begin{align*}
   \E\left[\left(\sum_{i=1}^k(\E[X^{(n)}_i|\cy_{i-1}] -
      \frac{1}{{\gamma}})\right)^2\right] 
&\leq
\E\left[\left(\sum_{i=1}^k|\E[X^{(n)}_i|\cy_{i-1}] -
      \frac{1}{{\gamma}}|\right)^2\right] \\
&\leq
  \E\left[\left(\sum_{i=1}^k C_{\ref{eq:M1}}
    \varphi_{Y_{i-1}} \right)^2\right]\\
&\leq  c \; \left(\sum_{j=1}^n
\varphi_j \right)^2  \leq c\; n^{3-\alpha},
\end{align*}
where for the last inequality we used \reff{eq:fn} with $\varepsilon_0>0$
small enough (such that $1+2\varepsilon_0<\alpha$) and the fact that 
$\zeta>1-1/\alpha$ implies $2-2\zeta\leq 3-\alpha$ as $\alpha\in
(1,2)$. 
This implies that $\displaystyle \E[|R_{n,1,k}|] 
\leq  c\;  n^{1-\alpha}$ and therefore $\displaystyle \E[\val{R_{n,1}}]
\leq  c\; n^{3-2\alpha}$. 
In particular, we get that $(n^{-1+\alpha-1/\alpha} R_{n,1}, n\geq 1)$
converges in probability to $0$ since $-1+\alpha-1/\alpha +
3-2\alpha=-(\alpha-1)^2/\alpha  <0$
for  $\alpha>1$.

We now consider $R_{n,2}$.  Suppose that $k\leq \lfloor nt\rfloor -1$
satisfies $\Delta_{n,k} \geq 1-\varepsilon$ on $\{nt<\tau_n\}$. Then on
$\{nt<\tau_n\}$, we have 
\[
\Delta_{n,k+1}
=\Delta_{n,k}+\frac{X^{(n)}_{k+1}-\inv{\gamma}+\frac{\Delta_{n,k}}{\gamma}}
{n-(k+1)/{\gamma}}
\geq \Delta_{n,k}+\frac{X^{(n)}_{k+1}-\frac{\varepsilon}{\gamma}}
{n-(k+1)/{\gamma}}
\geq \Delta_{n,k}, 
\]
where we used that  $\gamma>\varepsilon$  for the first inequality 
and $X^{(n)}_{k+1} \geq 1$ for the last. In particular, on
$\{nt<\tau_n\}$, if $\Delta_{n,k} \geq 1-\varepsilon$ for some $k\leq
\lfloor nt \rfloor $, then  we have $\Delta_{n,\lfloor nt \rfloor} \geq
1-\varepsilon$. This implies  that 
$\displaystyle \ind_{\{nt<\tau_n\}} R_{n,2}
= \ind_{\{ \Delta_{n,  \lfloor nt \rfloor} \geq
  1-\varepsilon\}}\ind_{\{nt<\tau_n\}}  R_{n,2}$.
With the notations of Corollary~\ref{cor:cvVn}, we have 
\[
\{nt<\tau_n\} \cap \{ \Delta_{n,  \lfloor nt \rfloor} \geq
1-\varepsilon\}  \subset \{
V_t^{(n)} \geq  (1-\varepsilon)(n -\frac{\lfloor nt \rfloor}{\gamma} )
n^{-1/\alpha}\}\subset \{
n^{-1+1/\alpha} V_t^{(n)} \geq c\}, 
\]
and then for any $\varepsilon'>0$ 
\begin{align*}
   \P(n^{-1+\alpha-1/\alpha} |R_{n,2}|\geq \varepsilon', \;nt<\tau_n)
&= \P(\ind_{\{ \Delta_{n,  \lfloor nt
    \rfloor} \geq 
  1-\varepsilon\}}n^{-1+\alpha-1/\alpha} |R_{n,2}|\geq \varepsilon',
\;nt<\tau_n )\\
&\leq \P( \Delta_{n,  \lfloor nt \rfloor} \geq
  1-\varepsilon,, \;nt<\tau_n)\\
& \leq  \P(n^{-1+1/\alpha} V_t^{(n)} \geq c).
\end{align*}
Use the  convergence of $V^{(n)}_t$, see  Corollary \ref{cor:cvVn}, to
get that the right-hand side of the last inequality converges to $0$
as $n$ goes to infinity. 
Then notice that $\displaystyle 
   \P(n^{-1+\alpha-1/\alpha} |R_{n,2}|\geq \varepsilon', \;nt\geq
   \tau_n)
\leq     \P(nt\geq \tau_n)$ which converges to $0$ thanks to
\reff{eq:cvPtn}. 

Thus the following convergence in probability holds:
\begin{equation}
   \label{eq:cvRNP}
 n^{-1+\alpha-1/\alpha} R_{n}
\; \xrightarrow[n\rightarrow \infty ]{\P} \;0.
\end{equation}

We deduce from \reff{eq:L=IJR}, \reff{eq:cvPI}, \reff{eq:cvJNP} and
\reff{eq:cvRNP} that 
\begin{equation}
   \label{eq:cvLNd}
 n^{-1+\alpha-1/\alpha} \left( \hat L^{(n)}_t -n^{2-\alpha} v(t) \right)
\; \xrightarrow[n\rightarrow \infty ]{\text{(d)}} \;
\gamma\left[\int_0^t dr \; \left(\int_r^t (1-\frac{s}{\gamma})^{-\alpha}
  ds\right)^\alpha \right]^{1/\alpha} V_1.
\end{equation}
To conclude,  use \reff{eq:TLfdVt} to get that $\displaystyle
\gamma\left[\int_0^t dr 
 \left(\int_r^t (1-\frac{s}{\gamma})^{-\alpha} 
  ds\right)^\alpha \right]^{1/\alpha} V_1$ is distributed as

$\displaystyle \gamma \int_0^t  dV_r \int _r^t
  (1-\frac{s}{\gamma})^{-\alpha} ds$
which in turn is equal to $\displaystyle  \int_0^t dr\;
(1-\frac{r}{\gamma})^{-\alpha} V_r$.

\end{proof}
\section{Proof of the main result}
\label{sec:main}
Let $\displaystyle \alpha_0=  \frac{1+\sqrt{5}}{2}$. Notice that for
$\alpha\in (1, \alpha_0)$, we have
$-1+\alpha-1/\alpha<0$, whereas for $\alpha\geq \alpha_0$,
$-1+\alpha-1/\alpha\geq 0$. Recall $\gamma=\alpha-1$. We define $a(t)$ for $t\in [0,\gamma]$ by 
\[
a(t)=\frac{v(t)}{C_0\Gamma(2-\alpha)},\quad\text{where}\quad v(t) =\int_0^t
    \left(1-\frac{r}{\gamma}\right)^{-\gamma} dr.
\]
We also set $\displaystyle V^*_t = \frac{\alpha-1 }{C_0
  \Gamma(2-\alpha)} \int_0^t  
(1-\frac{r}{\gamma})^{-\alpha} V_r\; dr$ for $t\in (0,\gamma)$. 
\begin{theo}
\label{th:ConvLn}   
We  assume   that $\rho(t)=C_0t^{-\alpha}
+O(t^{-\alpha+\zeta})$  for some 
  $C_0>0$ and $\zeta>1-1/\alpha$.
Then for all $t\in (0,\gamma)$, we have    that 
\begin{enumerate}
   \item The following convergence in probability holds:
\begin{equation}
   \label{eq:CVlNPv}
n^{-2+\alpha}  L^{(n)}_t \; \xrightarrow[n\rightarrow \infty ]{\P}
\;  a(t).
\end{equation}
   \item  If $\displaystyle \alpha\in (1, \alpha_0)$, the following convergence in distribution  holds:
\begin{equation}
   \label{eq:CVlNDV}
n^{-1+\alpha -1/\alpha} \left(  L^{(n)}_t- a(t)n^{2-\alpha} \right)  \; \xrightarrow[n\rightarrow \infty
]{\text{(d)}} 
\; V^*_t.   
\end{equation}
   \item  If $\displaystyle \alpha\in [\alpha_0,2)$, the following convergence in probability  holds: If ${\varepsilon}>0$, 
\begin{equation}
   \label{eq:CVlNDV2}
n^{-{\varepsilon}}\left(  L^{(n)}_t- a(t)n^{2-\alpha} \right)  \; \xrightarrow[n\rightarrow \infty]{\P} \; 0.   
\end{equation}
\end{enumerate}
\end{theo}

\begin{proof}
First of all, let us consider the case $\displaystyle \alpha\in (1,
\alpha_0)$.  
  Lemma  \ref{lem:Ltilde}  and  Tchebychev  inequality  imply  that  for
  $\alpha\in  (1,\alpha_0)$,  we   have  the  following  convergence  in
  probability
\[
\lim_{n\rightarrow \infty } n^{-1+\alpha-1/\alpha} |L^{(n)}_t -\tilde
L^{(n)}_t |=0.
\]
This and Lemma \ref{lem:Lhat} imply  that  for
  $\alpha\in  (1,\alpha_0)$,  we   have  the  following  convergence  in
  probability
\[
\lim_{n\rightarrow \infty } n^{-1+\alpha-1/\alpha} |L^{(n)}_t -\frac{\hat 
L^{(n)}_t}{C_0 \Gamma(2-\alpha)} |=0.
\]
The result is then  a direct consequence of Theorem \ref{theo:main}.

For $\displaystyle \alpha\in [\alpha_0,2)$, note that ${\alpha}>3/2$ and
$-1+{\alpha}-1/{\alpha}\geq     0$.     As     $\zeta>1-1/\alpha$    and
$\alpha>\alpha_0$  i.e. $1-1/\alpha>2-\alpha$, we  get $\zeta>2-\alpha$.
We  then use  Lemma  \ref{lem:Ltilde}, Lemma  \ref{lem:Lhat} (only  with
$\zeta>2-\alpha$) and Theorem  \ref{theo:main} to get \reff{eq:CVlNDV2},
and then \reff{eq:CVlNPv}.
\end{proof}
Let $K^{(n)}_t$ be the number of mutations up to the $\lfloor nt \rfloor$-th
coalescence, for $t\in(0,{\gamma})$. conditionally on $L^{(n)}_t$,
$K^{(n)}_t$ is a Poisson r.v. with parameter $\theta L^{(n)}_t$. 
The next Corollary is a consequence of Theorem~\ref{th:ConvLn}. 

\begin{cor}
\label{cor:main}
We assume  that $\rho(t)=C_0t^{-\alpha}  +O(t^{-\alpha+\zeta})$ for  some 
  $C_0>0$ and $\zeta>1-1/\alpha$.
Let  $t\in (0,\gamma)$ and $G$  be a standard Gaussian r.v., independent
of $V$. 
\begin{enumerate}
\item For  ${\alpha}\in(1,\sqrt 2)$,  we have 
\[
n^{-1+\alpha -1/\alpha} (K^{(n)}_t-{\theta}a(t)n^{2-\alpha})
\; \xrightarrow[n\rightarrow \infty]{\text{(d)}} \;  
 \theta V^*_t.
\]
\item For  ${\alpha}\in(\sqrt 2,2)$, we have
\[
n^{-1+\alpha/2} (K^{(n)}_t-{\theta}a(t)n^{2-\alpha}) 
\; \xrightarrow[n\rightarrow \infty]{\text{(d)}} \;  
\sqrt{{\theta}a(t) } G.
\]
\item For ${\alpha}=\sqrt 2$, we have
  $-1+\alpha-\inv{\alpha}=1-\frac{\alpha}{2}$ and 
\[
n^{-1+\alpha -1/\alpha} (K^{(n)}_t-{\theta}a(t)n^{2-\alpha})
\; \xrightarrow[n\rightarrow \infty]{\text{(d)}} \; \theta V^*_t
+ \sqrt{{\theta}a(t)}G.
\]
\end{enumerate}
\end{cor}
\begin{proof}
Let us compute the characteristic function ${\psi}_n(u,v)$ of the 2-dimensional r.v. $(G_n,H_n)$ with
\[
G_n=\frac{K^{(n)}_t-{\theta}L^{(n)}_t}{\sqrt{{\theta}a(t) n^{2-\alpha}}}\quad
\text{and}\quad 
H_n=n^{-1+\alpha -1/\alpha} \left(  L^{(n)}_t-a(t)n^{2-\alpha} \right).
\]
Using that, conditionally on $L^{(n)}_t$, the law of $K^{(n)}_t$ is a
Poisson distribution with parameter ${\theta}L^{(n)}_t$, we have  
\begin{align*}
{\psi}_n(u,v)
={\mathbb E}\left[\expp{iuG_n}\expp{ivH_n}\right]
&={\mathbb E}
\left[
\expp{
-{\theta}L^{(n)}_t\left(1-\expp{iu/\sqrt{{\theta}a(t)n^{2-\alpha}
  }}+iu/\sqrt{{\theta}a(t) n^{2-\alpha}}\right)
}
\expp{ivH_n}
\right].
\end{align*}

We first consider the case ${\alpha}\in (1,{\alpha}_0)$. 
Using    Theorem~\ref{th:ConvLn},    we    get    that    
\[
-{\theta}L^{(n)}_t\left(1-\expp{iu/\sqrt{{\theta}a(t)n^{2-\alpha}
  }}+iu/\sqrt{{\theta}a(t)  n^{2-\alpha}} \right) 
\]
tends  to   $-u^2/2$  in
probability  and   has  a  non-negative  real   part.   Hence,  applying
Theorem~\ref{th:ConvLn}  again,  we get  that  $(G_n,H_n)$ converges  in
distribution to  $(G,V^*_t)$, where  $G$ is a
standard  Gaussian   r.v. independent of $V$. Notice that
\[
K^{(n)}_t
=\theta a(t)n^{2-\alpha}+{\theta}n^{1-{\alpha}+1/{\alpha}}H_n+\sqrt{\theta
a(t)}n^{1-\alpha/2}G_n.
\]
We have $\sqrt{2}<\alpha_0$. To conclude when ${\alpha}<{\alpha}_0$, use
that $1-\alpha+1/\alpha$ is smaller (resp. equal to) $1-\alpha/2$ if and
only if $\alpha>\sqrt{2}$ (resp. $\alpha=\sqrt{2}$).

Now we consider  ${\alpha}\in [{\alpha}_0,2)$. We write 
\[
n^{-1+\alpha/2} (K^{(n)}_t-{\theta}a(t)n^{2-\alpha})
=
\sqrt{\theta a(t)} G_n + n^{-1+\alpha/2}(L^{(n)}_t-a(t)n^{2-{\alpha}}).
\]
Using Theorem~\ref{th:ConvLn}, we still get that $G_n$ converges in law
to $G$. Moreover, \reff{eq:CVlNDV2} implies that $
n^{-1+\alpha/2}(L^{(n)}_t-a(t)n^{2-{\alpha}})$ converges to 0 in
probability. This gives the result.
\end{proof}



\begin{thebibliography}{10}

\bibitem{bbs:bcsrt}
J.~Berestycki, N.~Berestycki, and J.~Schweinsberg.
\newblock Beta-coalescents and continuous stable random trees.
\newblock {\em Ann. Probab.}, To appear.

\bibitem{bbs:stpbc}
J.~Berestycki, N.~Berestycki, and J.~Schweinsberg.
\newblock Small time properties of beta-coalescents.
\newblock {\em Ann. Inst. H. Poincaré Probab. Statist.}, To appear.

\bibitem{b:pl}
J.~Bertoin.
\newblock {\em L\'evy processes}.
\newblock Cambridge University Press, Cambridge, 1996.

\bibitem{blg:sfcp3}
J.~Bertoin and J.-F. Le~Gall.
\newblock Stochastic flows associated to coalescent processes. {III}. {L}imit
  theorems.
\newblock {\em Illinois J. Math.}, 50(1-4):147--181 (electronic), 2006.

\bibitem{bbcemsw}
M.~Birkner, J.~Blath, M.~Capaldo, A.~Etheridge, M.~M{\"o}hle, J.~Schweinsberg,
  and A.~Wakolbinger.
\newblock Alpha-stable branching and beta-coalescents.
\newblock {\em Electron. J. Probab.}, 10:no. 9, 303--325 (electronic), 2005.

\bibitem{bs:rpcacm}
E.~Bolthausen and A.-S. Sznitman.
\newblock On {R}uelle's probability cascades and an abstract cavity method.
\newblock {\em Comm. Math. Phys.}, 197(2):247--276, 1998.

\bibitem{bbb:mdviipopo}
E.~G. Boom, J. D. G.and~Boulding and A.~T. Beckenbach.
\newblock Mitochondrial dna variation in introduced populations of pacific
  oyster, crassostrea gigas, in british columbia.
\newblock {\em Can. J. Fish. Aquat. Sci.}, 51:1608--1614, 1994.

\bibitem{b:p}
L.~Breiman.
\newblock {\em Probability}, volume~7 of {\em Classics in Applied Mathematics}.
\newblock Society for Industrial and Applied Mathematics (SIAM), Philadelphia,
  PA, 1992.
\newblock Corrected reprint of the 1968 original.

\bibitem{dimr:rrc}
M.~Drmota, A.~Iksanov, M.~M{\"o}hle, and U.~Rösler.
\newblock Asymptotic results concerning the total branch length of the
  {B}olthausen-{S}znitman coalescent.
\newblock {\em Stoch. Process. Appl.}, 117, To appear.

\bibitem{ew:cpwtdoonaiihs}
B.~Eldon and J.~Wakeley.
\newblock Coalescent processes when the distribution of offspring number among
  individuals is highly skewed.
\newblock {\em Genetics}, 172:2621--2633, 2006.

\bibitem{f:ipta}
W.~Feller.
\newblock {\em An introduction to probability theory and its applications},
  volume~II.
\newblock John WILEY \& Sons, 1971.

\bibitem{gy:ncl}
A.~Gnedin and Y.~Yakubovich.
\newblock On the number of collisions in $\llambda$-coalescents, 2007.

\bibitem{im:rrratdmc}
A.~Iksanov and M.~M{\"o}hle.
\newblock On a random recursion related to absorption times of death {M}arkov
  chains, 2007.

\bibitem{k:fmp}
O.~Kallenberg.
\newblock {\em Foundations of modern probability}.
\newblock Probability and its Applications (New York). Springer-Verlag, New
  York, second edition, 2002.

\bibitem{k:c}
J.~F.~C. Kingman.
\newblock The coalescent.
\newblock {\em Stochastic Process. Appl.}, 13(3):235--248, 1982.

\bibitem{k:p}
J.~F.~C. Kingman.
\newblock Origins of the coalescent 1974--1982.
\newblock {\em Genetics}, 156:1461--1463, 2000.

\bibitem{m:nssplfs}
M.~M{\"o}hle.
\newblock On the number of segregating sites for populations with large family
  sizes.
\newblock {\em Adv. in Appl. Probab.}, 38(3):750--767, 2006.

\bibitem{mrs:nmgf}
A.~Mukherjea, M.~Rao, and S.~Suen.
\newblock A note on moment generating functions.
\newblock {\em Statist. Probab. Lett.}, 76(11):1185--1189, 2006.

\bibitem{p:cmc}
J.~Pitman.
\newblock Coalescents with multiple collisions.
\newblock {\em Ann. Probab.}, 27(4):1870--1902, 1999.

\bibitem{s:gcamal}
S.~Sagitov.
\newblock The general coalescent with asynchronous mergers of ancestral lines.
\newblock {\em J. Appl. Probab.}, 36(4):1116--1125, 1999.

\bibitem{s:cposgwp}
J.~Schweinsberg.
\newblock Coalescent processes obtained from supercritical {G}alton-{W}atson
  processes.
\newblock {\em Stochastic Process. Appl.}, 106(1):107--139, 2003.

\bibitem{w:nssgmwr}
G.~A. Watterson.
\newblock On the number of segregating sites in genetical models without
  recombination.
\newblock {\em Theoret. Population Biology}, 7:256--276, 1975.

\end{thebibliography}

\end{document}